\documentclass[11pt,oneside,reqno]{amsart}

\usepackage[a4paper, total={450pt,675pt}]{geometry}
\usepackage[OT2,T1]{fontenc}
\usepackage[utf8]{inputenc}
\usepackage{amssymb,bm}
\usepackage{amsmath} 
\usepackage{csquotes}

\usepackage{float} 

\usepackage[dvipsnames]{xcolor}
\usepackage[
colorlinks=true,
linkcolor=Maroon,
citecolor=JungleGreen,
urlcolor=NavyBlue, linktocpage]{hyperref}

\usepackage[capitalize,nameinlink]{cleveref}

\usepackage{enumitem}
\usepackage{dsfont}
\usepackage{subcaption}
\usepackage{centernot}

\usepackage{yhmath} 

\usepackage{tikz}
\usetikzlibrary{arrows}
\usetikzlibrary{decorations.pathreplacing,angles,quotes}
\usepackage{pgfplots}
\pgfplotsset{compat=1.15}
\usepackage{mathrsfs}
\usetikzlibrary{arrows}

\usepackage{soul}
\setstcolor{red}

\makeatletter
\def\@tocline#1#2#3#4#5#6#7{\relax
	\ifnum #1>\c@tocdepth 
	\else
	\par \addpenalty\@secpenalty\addvspace{#2}%
	\begingroup \hyphenpenalty\@M
	\@ifempty{#4}{%
		\@tempdima\csname r@tocindent\number#1\endcsname\relax
	}{%
		\@tempdima#4\relax
	}%
	\parindent\z@ \leftskip#3\relax
	\advance\leftskip\@tempdima\relax
	\rightskip\@pnumwidth plus4em \parfillskip-\@pnumwidth
	#5\leavevmode\hskip-\@tempdima
	\ifcase #1
	\or\or \hskip 2em \or \hskip 2em \else \hskip 3em \fi%
	#6\nobreak\relax
	\dotfill\hbox to\@pnumwidth{\@tocpagenum{#7}}\par
	\nobreak
	\endgroup
	\fi} 
\makeatother

\usepackage[
backend=bibtex,
style=alphabetic,
maxbibnames=10,
sorting=nyt]{biblatex}

\DeclareLabelalphaTemplate{
	\labelelement{
		\field[final]{shorthand}
		\field{label}
		\field[strwidth=3]{labelname}}
	\labelelement{
		\field[strwidth=2,strside=right]{year}}
}

\DeclareFieldFormat[article,book,incollection,misc]{title}{#1}
\DeclareFieldFormat[inbook,incollection]{booktitle}{#1}
\DeclareFieldFormat[misc]{date}{\textit{preprint} (#1)}
\DeclareFieldFormat{pages}{#1}
\DeclareFieldFormat{doi}{
	\href{http://www.ams.org/mathscinet-getitem?mr=#1}{#1}}

\renewbibmacro{in:}{%
	\ifentrytype{article}{}{\printtext{\bibstring{in}\intitlepunct}}}

\newtheorem{theorem}{Theorem}[section]

\newtheorem{lemma}[theorem]{Lemma}
\newtheorem{proposition}[theorem]{Proposition}
\newtheorem{corollary}[theorem]{Corollary}

\newtheorem{remark}[theorem]{Remark}

\crefname{section}{Sect.}{section}
\numberwithin{equation}{section}

\newcommand*\diff{\mathop{}\!\mathrm{d}}

\DeclareMathOperator{\supp}{supp}

\allowdisplaybreaks

\begin{document}
	\title[$L^p$ heat asymptotics on symmetric spaces for non-symmetric solutions]{$L^p$ asymptotics for the heat equation on symmetric spaces for non-symmetric solutions}
	
	\author{Effie Papageorgiou}
	
	\begin{abstract}
		The main goal of this work is to study the $L^p$-asymptotic behavior of
		solutions to the heat equation on arbitrary rank Riemannian
		symmetric spaces of non-compact type $G/K$ for non-bi-$K$ invariant initial data. For initial data $u_0$ compactly supported or in a weighted $L^1(G/K)$ space with a weight depending on $p\in [1, \infty]$, we introduce a mass function $M_p(u_0)(\cdot)$, and prove that if $h_t$ is the heat kernel on $G/K$, then 
		$$\|h_t\|_p^{-1}\,\|u_0\ast h_t \, - \,M_p(u_0)(\cdot)\,h_t\|_p \rightarrow 0 \quad \text{as} \quad t\rightarrow \infty.$$
		Interestingly, the $L^p$ heat concentration leads to completely different expressions of the mass function for $1\leq p <2$ and $2\leq p\leq \infty$. If we further assume that the initial data are bi-$K$-invariant, then our mass function boils down to the constant $\int_{G/K}u_0$ in the case $p=1$, and more generally to $\mathcal{H}{u_0}(i\rho(2/p-1))$ if $1\leq p<2$, and to $\mathcal{H}{u_0}(0)$ if $2\leq p \leq \infty$. Thus we improve upon results by V\'azquez, Anker et al, Naik et al,  clarifying the nature of the problem. 
	\end{abstract}
	

	\makeatletter
	\@namedef{subjclassname@2020}{\textnormal{2020}
		\it{Mathematics Subject Classification}}
	\makeatother
	
	\maketitle
	\tableofcontents

	\section{Introduction}\label{Section.1 Intro}

The heat equation is a fundamental partial differential equation in mathematics. Since Joseph Fourier's renowned work in 1822, it has been extensively studied across various contexts and is recognized as a key element in multiple areas of mathematics (see, for example, \cite{Gri2009}). The following classical result on long-time asymptotic convergence is analogous to the Central Limit Theorem in probability, but within the framework of partial differential equations. For more details on this property, we refer to the expository survey \cite{Vaz2018}.

\begin{theorem}
	Consider the heat equation
	\begin{align}\label{S1 HE intro}
		\begin{cases}
			\partial_{t}u(t,x)\,
			=\,\Delta_{\mathbb{R}^{n}}u(t,x),
			\qquad\,t>0,\,\,x\in\mathbb{R}^{n}\\[5pt]
			u(0,x)\,=\,u_{0}(x),
		\end{cases}
	\end{align}
	where the initial data $u_{0}$ belongs to $L^{1}(\mathbb{R}^n)$.
	Denote by $M=\int_{\mathbb{R}^n}\diff{x}\,u_{0}(x)$ the mass of $u_{0}$
	and by $G_{t}(x)\,=\,(4\pi{t})^{-n/2}e^{-|x|^{2}/4t}$ the heat kernel.
	Then the solution to \eqref{S1 HE intro} satisfies:
	\begin{align}\label{S1 L1 R}
		\|u(t,\,\cdot\,)-MG_{t}\|_{L^{1}(\mathbb{R}^n)}\,
		\longrightarrow\,0
	\end{align}
	and
	\begin{align}\label{S1 Linf R}
		t^{\frac{n}{2}}\,\|u(t,\,\cdot\,)-MG_{t}\|_{L^{\infty}(\mathbb{R}^n)}\,
		\longrightarrow\,0
	\end{align}
	as $t\rightarrow\infty$. The $L^p$ ($1<p<\infty$) norm estimates 
	follow by interpolation:
	\begin{align}\label{S1 Lp R}
		t^{\frac{n}{2p'}}\,\|u(t,\,\cdot\,)-MG_{t}\|_{L^{p}(\mathbb{R}^n)}\,
		\longrightarrow\,0,
	\end{align}
	where $\|G_t\|_{L^p(\mathbb{R}^n)}\asymp t^{\frac{n}{2p'}}$, $1\leq p \leq \infty$. Here, $p'$ denotes the dual exponent of $p$, defined by the formula
	$\tfrac{1}{p}+\tfrac{1}{p'}=1$.
\end{theorem}

On manifolds of non-negative Ricci curvature, the results \eqref{S1 Lp R} were generalized in \cite{GPZ2022}. Analogous properties on negatively curved manifolds
were first investigated by Vázquez in his recent work \cite{VazHyp},
which deals with real hyperbolic spaces $\mathbb{H}^n(\mathbb{R})$. More precisely, it was shown in \cite{VazHyp} that on $\mathbb{H}^3(\mathbb{R})$, (\ref{S1 L1 R}) fails for general absolutely integrable initial
data $u_0$ but is still true on $\mathbb{H}^n(\mathbb{R})$ if the function $
u_0$ is spherically symmetric around the origin. A generalization of these results was
obtained in \cite{APZ2023} in the more general setting of symmetric spaces of
non-compact type by using tools of harmonic analysis. Let us elaborate.
Given a symmetric space $\mathbb{X}=G/K$ of non compact type,
let us denote by $\Delta$ the Laplace-Beltrami operator on $\mathbb{X}$
and by $h_t$ the associated heat kernel, i.e.,
the bi-$K$-invariant convolution kernel of the semi-group $e^{t\Delta}$. The heat equation writes as follows:
\begin{align}\label{S1 HE X}
	\partial_{t}u(t,x)\,=\,\Delta_{x}u(t,x),
	\qquad
	u(0,x)\,=\,u_{0}(x), \quad x\in \mathbb{X}, \, t>0.
\end{align}
\begin{theorem}(\cite{APZ2023})\label{S1 Main thm 1 APZ}
	Let $u_{0}\in{L^{1}(\mathbb{X})}$ be a bi-$K$-invariant initial condition
	and $M=\int_{\mathbb{X}}\diff{x}\,u_{0}(x)$ be its mass. Then the solution to
	the heat equation \eqref{S1 HE X} satisfies
	\begin{align}\label{S1 Main thm 1 convergence APZ}
		\|u(t,\,\cdot\,)-Mh_{t}\|_{L^{1}(\mathbb{X})}\,
		\longrightarrow\,0
		\qquad\textnormal{as}\quad\,t\rightarrow\infty.
	\end{align}
	Moreover, this convergence fails in general without the bi-$K$-invariance assumption.
	In addition, for $p>1$ and $u_0\in L^1(\mathbb{X})$, the following hold
	\begin{align}
		\|u(t,\,\cdot\,)-Mh_{t}\|_{L^{\infty}(\mathbb{X})}\,
		&=\,
		\mathrm{O}\big(t^{-\frac{\nu}{2}}e^{-|\rho|^{2}t}\big) \label{LinftyX},  \\[5pt]
		\|u(t,\,\cdot\,)-Mh_{t}\|_{L^{p}(\mathbb{X})}\,
		&=\,\mathrm{o}\big(t^{-\frac{\nu}{2p'}}
		e^{-\tfrac{|\rho|^{2}t}{p'}}\big) \qquad \, 1<p<\infty, \quad \text{if} \quad u_0 \quad \text{is bi-$K$-invariant,} \label{LpX}
	\end{align}
	as $t\rightarrow\infty$.  Here $\nu$ is the so-called dimension at infinity of $\mathbb{X}$ and $|\rho|^2$ coincides with the bottom of the $L^2$ spectrum of the Laplace-Beltrami operator.
\end{theorem}

Interestingly, these results for $p>1$ are not of Euclidean flavor, in the sense of convergence in \eqref{S1 Lp R}; more precisely, given that $\|h_t\|_{\infty}\asymp  t^{-\frac{\nu}{2}}e^{-|\rho|^{2}t}$ for $t>1$, it follows from \eqref{LinftyX} that $\|h_t\|_{\infty}^{-1} \,\|u(t,\,\cdot\,)-\,M \,h_{t}\|_{L^{\infty}(\mathbb{X})}\centernot\longrightarrow\,0 $ (see \cite[Remark 3.6]{APZ2023} for details). Even further, the right hand side of \eqref{LpX} is not comparable to $\|h_t\|_p$ for any $p\in (1, \infty)$. This prompted Naik et al in \cite{NRS24} to seek, instead of the mass $M$, a complex number $z$ (which turns out to be unique, if one restricts to a certain class of bi-$K$-invariant initial data), so that for appropriate bi-$K$-invariant data one has 
$$\|h_t\|_p^{-1}	\, \|u(t,\,\cdot\,)\, - z\,h_{t}\|_{L^{p}(\mathbb{X})}\,
\longrightarrow\,0
\qquad\textnormal{as}\quad\,t\rightarrow\infty.$$ 
Let us state their result. We say that a function $u_0:\mathbb{X}\rightarrow \mathbb{C}$ belongs to $\mathcal{L}^p(G//K)$,  $p\in [1, \infty]$, if $u_0$ is bi-$K$-invariant and 
$$\int_{G}\textrm{d}x\, |u_0(x)|\, \varphi_{i(\frac{2}{p}-1)\rho}(x)<+\infty, \qquad \text{if} \quad 1\leq p<2, $$
and 
$$\int_{G}\textrm{d}x\, |u_0(x)|\, \varphi_{0}(x)<+\infty, \qquad \text{if} \quad 2\leq p\leq \infty, $$
where $\varphi_{\lambda}$ denotes the elementary spherical function of index $\lambda$.
Then, define accordingly 
$$M_p:=\int_{G}\textrm{d}x\, u_0(x)\, \varphi_{i(\frac{2}{p}-1)\rho}(x)=\mathcal{H}u_0\left(i(2/p-1)\rho\right)\qquad\text{if} \quad 1\leq p<2, $$
and
$$M_p:=\int_{G}\textrm{d}x\, u_0(x)\, \varphi_{0}(x)=\mathcal{H}u_0(0), \qquad \text{if} \quad 2\leq p\leq \infty. $$

\begin{theorem} (\cite{NRS24}) 
	Let $u_{0}\in \mathcal{L}^p(G//K)$
	and $M_p$ be defined as above. Then the solution to
	the heat equation \eqref{S1 HE X} with initial condition $u_0$ satisfies
	\begin{align}\label{Ind convergence}
		\|h_t\|_p^{-1}	\, \|u(t,\,\cdot\,)\,- \,M_p\,h_{t}\|_{L^{p}(\mathbb{X})}\,
		\longrightarrow\,0
		\qquad\textnormal{as}\quad\,t\rightarrow\infty.
	\end{align}
\end{theorem}

Observe that when $u_0\in \mathcal{L}^{1}(G//K)={L}^{1}(G//K)$, $M_1=\int_{G}\textrm{d}x\, u_0(x)$.
Thus the convergence \eqref{Ind convergence} agrees with the bi-$K$-invariant $L^1$ result 
previously obtained in \cite{VazHyp} on $\mathbb{H}^{n}(\mathbb{R})$, as well as the bi-$K$-invariant $L^1$ results of \cite{APZ2023} on arbitrary rank symmetric spaces and gives an analogue of the Euclidean $L^p$ convergence in the spirit of \cite{Vaz2018}. 

However, the result above shows $L^p$ convergence restricted to \textit{bi-$K$-invariant} initial data. In the non-bi-$K$-invariant case, V{\'a}zquez showed  in three-dimensional hyperbolic space that \eqref{Ind convergence} for $p=1$ fails for non-radial initial data, and we established in \cite{APZ2023} such a counterexample for $p=1$ not only for $\mathbb{H}^{n}(\mathbb{R})$,
$n\neq3$, but also for the symmetric spaces of general rank. Here, we illuminate this phenomenon by pushing the strong $L^p$ convergence of \eqref{Ind convergence} to \textit{non-bi-$K$-invariant initial data}. To do this, we replace the constant $M_p$ by a \textit{function} $M_p(u_0)(gK)$, $g\in G$, which, however, boils down to the constants above, if the initial datum is bi-$K$-invariant. In this way, the counterexamples mentioned above, obtained simply by displaced heat kernels $h_t(yK, \cdot K)$, $y\notin K$, now obey convergence. Our proof is different and shorter than the one pursued in \cite{NRS24}, which follows ideas from \cite{APZ2023} where the main work was done on the \enquote{frequency} side; instead, our approach mainly relies on heat kernels asymptotics, eventually obtaining more information this way.

To introduce our results, let us introduce some notation. Recall the Cartan decomposition $G=K\overline{A^{+}}K$ and the Iwasawa decomposition $G=NAK$ of the group $G$, according to which every $g\in G$ can be written as 
$g=k(g)(\exp g^{+})k'(g)$ and $g=n(g)(\exp A(g))\kappa(g)$, respectively.
Consider first the  weights $w_p$ on $\mathbb{X}$:
$$w_p(gK):=e^{\frac{2}{p}\langle \rho, g^{+} \rangle},  \quad 1\leq p \leq 2, \quad \text{and} \quad w_p(gK):=e^{\langle \rho, g^{+} \rangle}, \quad 2\leq p\leq \infty.$$
Next, consider the following weighted $L^1$ space defined by 
$$L_{w_p}(\mathbb{X}):=\left\{u_0:\mathbb{X}\rightarrow \mathbb{C}: \,  \int_{G}\textrm{d}x\, |u_0(x)|\, w_p(x)<+\infty\right\}.$$
Define also the mass functions by 
$$M_p(u_0)(g):=\int_G \diff{y}\,u_{0}(yK)\,e^{\frac{2}{p}\,\langle \rho , A(k^{-1}(g)y)\rangle}, \qquad 1\leq p\leq 2$$ 
and by
$$M_p(u_0)(g):=\frac{(u_{0}*\varphi_{0})(g)}{\varphi_{0}(g)}, \qquad 2\leq p \leq \infty,$$
(any of the two masses can be used in the critical case $p=2$).
Then our main result is the following.

\begin{theorem}\label{S1 Main thm 2}
	Let $p\in[1, \infty]$, and let $u_{0}\in \mathcal{C}_c(\mathbb{X})$, $u_0\in \mathcal{L}^p(G//K)$ or $u_0\in L_{w_p}(\mathbb{X})\subseteq L^1(\mathbb{X})$.
	Consider $M_p(u_0)(\cdot)$ as above. Then the solution to
	the heat equation \eqref{S1 HE X} satisfies
	\begin{align}\label{S1 Main thm 2 convergence}
		\|h_t\|_p^{-1}	\, \|u(t,\,\cdot\,)-M_p(u_0)(\cdot)\,h_{t}\|_{L^{p}(\mathbb{X})}\,
		\longrightarrow\,0
		\qquad\textnormal{as}\quad\,t\rightarrow\infty.
	\end{align}
	In the bi-$K$-invariant case, the mass functions boil down to constants, that is $M_p(u_0)(gK)= \linebreak \mathcal{H}u_0\left(i(2/p-1)\rho\right)$ if $ 1\leq p\leq 2$,  $M_p(u_0)(gK)=\mathcal{H}u_0(0)$ if $ 2\leq p \leq \infty$, and  $L^1(G//K)\subseteq \mathcal{L}^p(G//K)$, $1\leq p \leq \infty$.
	
\end{theorem}
To the best of our knowledge this is the first affirmative result, even for the case of hyperbolic space, dealing with the long-time convergence of the solution to the heat equation beyond the bi-$K$-invariant case. A mass function was considered in \cite{APZ2023} but only for the distinguished Laplacian. Clearly, if $u_{0}\in \mathcal{C}_c(\mathbb{X})$ then $u_0\in L_{w_p}(\mathbb{X})$, but we want to emphasize that all previously known non-symmetric solutions which served as counterexamples, namely displaced heat kernels, now obey convergence. It is an interesting non-Euclidean phenomenon that the mass should be a function rather than a constant, and that the $L^p$ propagation of the heat kernel results in dramatically different behaviors of the mass functions above and below the critical point $p=2$.  For the interested reader who is not familiar with higher rank symmetric spaces, we give  explicit computations of these mass functions and weights in the case of real hyperbolic space in the Appendix.  Our \cref{S1 Main thm 2} is a generalization of \eqref{Ind convergence} and of the main results in \cite{VazHyp, APZ2023, NRS24} for symmetric initial data, and settles the problem as it clarifies the long time behavior of solutions
which are not necessarily
symmetric.

Unlike the approach (mainly) pursued in \cite{APZ2023} or \cite{NRS24}, we do not work on the \enquote{frequency side}, that is, with the use of the spherical transform or the Helgason-Fourier transform (which, in the case that the mass is a function, would be hard to handle). Instead, we work directly on $\mathbb{X}$ with a careful analysis of the heat kernel in the so-called $L^p$ critical region (for all notions and notations, we refer the reader to Section \ref{Section.2 Prelim}).
On the other hand, whether the result of \cref{S1 Main thm 2}  holds true for all $L^{1}(\mathbb{X})$ initial data is still an open question for further study.

This paper is organized as follows. After the present introduction
in \cref{Section.1 Intro} and preliminaries in \cref{Section.2 Prelim}, 
we discuss $L^p$ critical regions and mass functions in 
\cref{Section.3 lp critical}. 
In \cref{Section.4 lp 1 to 2}, we prove our main results related to the $L^p$ convergence of solutions to the heat equation for $1\leq p <2$, while we in Section \cref{Section.5 lp 1 to infty} we treat the range $2\leq p\leq \infty$. An Appendix is given in the end with explicit computations for the real hyperbolic space case. 

Throughout this paper, the notation
$A\lesssim{B}$ between two positive expressions means that 
there is a constant $C>0$ such that $A\le{C}B$. 
The notation $A\asymp{B}$ means that $A\lesssim{B}$ and $B\lesssim{A}$. Also, $A(t)\sim B(t)$ means that $A(t)/B(t)\rightarrow 1$ as $t\rightarrow +\infty$.

\section{Preliminaries}\label{Section.2 Prelim}
In this section, we first review spherical Fourier analysis 
on Riemannian symmetric spaces of non-compact type. The notation is standard 
and follows \cite{Hel1978,Hel2000,GaVa1988}.
Next we recall the asymptotic concentration of the heat kernel.
We refer to \cite{AnJi1999,AnOs2003} for more details on the heat kernel 
analysis in this setting.
\subsection{Non-compact Riemannian symmetric spaces}
Let $G$ be a semi-simple Lie group, connected, non-compact, with finite center, 
and $K$ be a maximal compact subgroup of $G$. The homogeneous space 
$\mathbb{X}=G/K$ is a Riemannian symmetric space of non-compact type.
Let $\mathfrak{g}=\mathfrak{k}\oplus\mathfrak{p}$ be the Cartan decomposition 
of the Lie algebra of $G$. The Killing form of $\mathfrak{g}$ induces 
a $K$-invariant inner product $\langle\,.\,,\,.\,\rangle$ on $\mathfrak{p}$, 
hence a $G$-invariant Riemannian metric on $G/K$.
We denote by $d(\,.\,,\,.\,)$ the Riemannian distance on $\mathbb{X}$.

Fix a maximal abelian subspace $\mathfrak{a}$ in $\mathfrak{p}$. 
The rank of $\mathbb{X}$ is the dimension $\ell$ of $\mathfrak{a}$.
We identify $\mathfrak{a}$ with its dual $\mathfrak{a}^{*}$ 
by means of the inner product inherited from $\mathfrak{p}$. 	Sometimes we shall use coordinates on $\mathfrak{a}$; when we do, we shall always refer to the
coordinates associated to the orthonormal basis $\delta_1, ... , \delta_{\ell-1}, \,\rho/|\rho|$, where $ \delta_1, ... , \delta_{\ell-1},$ is any orthonormal basis of $\rho^{\perp}.$

Let $\Sigma\subset\mathfrak{a}$ be the root system of 
$(\mathfrak{g},\mathfrak{a})$ and denote by $W$ the Weyl group 
associated with $\Sigma$. 
Once a positive Weyl chamber $\mathfrak{a}^{+}\subset\mathfrak{a}$ 
has been selected, $\Sigma^{+}$ (resp. $\Sigma_{r}^{+}$ 
or $\Sigma_{s}^{+}$)  denotes the corresponding set of positive roots 
(resp. positive reduced, i.e., indivisible roots or simple roots).
Let $n$ be the dimension and $\nu$ be the pseudo-dimension 
(or dimension at infinity) of $\mathbb{X}$: 
\begin{align}\label{S2 Dimensions}
	\textstyle
	n\,=\,
	\ell+\sum_{\alpha \in \Sigma^{+}}\,m_{\alpha}
	\qquad\textnormal{and}\qquad
	\nu\,=\,\ell+2|\Sigma_{r}^{+}|
\end{align}
where $m_{\alpha}$ denotes the dimension of the positive root subspace
\begin{align*}
	\mathfrak{g}_{\alpha}\,
	=\,\lbrace{
		X\in\mathfrak{g}\,|\,[H,X]=\langle{\alpha,H}\rangle{X},\,
		\forall\,H\in\mathfrak{a}
	}\rbrace.
\end{align*}

Let $\mathfrak{n}$ be the nilpotent Lie subalgebra 
of $\mathfrak{g}$ associated with $\Sigma^{+}$ 
and let $N = \exp \mathfrak{n}$ be the corresponding 
Lie subgroup of $G$. We have the decompositions 
\begin{align*}
	\begin{cases}
		\,G\,=\,N\,(\exp\mathfrak{a})\,K 
		\qquad&\textnormal{(Iwasawa)}, \\[5pt]
		\,G\,=\,K\,(\exp\overline{\mathfrak{a}^{+}})\,K
		\qquad&\textnormal{(Cartan)}.
	\end{cases}
\end{align*}
Denote by $A(x)\in\mathfrak{a}$ and $x^{+}\in\overline{\mathfrak{a}^{+}}$
the middle components of $x\in{G}$ in these two decompositions, and by
$|x|=|x^{+}|$ the distance to the origin. For all $x,y\in{G}$, we have
\begin{align}\label{S2 Distance}
	|A(xK)|\,\le\,|x|
	\qquad\textnormal{and}\qquad
	|x^{+}-y^{+}|\,\le\,d(xK,yK),
\end{align}
see for instance \cite[Lemma 2.1.2]{AnJi1999}. For an element $x\in G$, another way to compare its middle components $A(x)$ in the Iwasawa decomposition and $x^{+}$ the Cartan decomposition, is the following consequence of Kostant's convexity lemma, \cite[Lemma 4.8]{APZ2023}: 
\begin{align}\label{S2 ineq Iwasawa Cartan}
	\langle{\rho,A(x)}\rangle\,
	\le\,\langle{\rho,x^{+}}\rangle.
\end{align}	

In the Cartan decomposition, the Haar measure 
on $G$ writes
\begin{align*}
	\int_{G}\diff{x}\,u_0(x)\,
	=\,
	|K/{\mathbb{M}}|\,\int_{K}\diff{k_1}\,
	\int_{\mathfrak{a}^{+}}\diff{x^{+}}\,\delta(x^{+})\, 
	\int_{K}\diff{k_2}\,u_0(k_{1}(\exp x^{+})k_{2})\,,
\end{align*}
with density
\begin{align}\label{S2 estimate of delta}
	\delta(x^{+})\,
	=\,\prod_{\alpha\in\Sigma^{+}}\,
	(\sinh\langle{\alpha,x^{+}}\rangle)^{m_{\alpha}}\,
	\asymp\,
	\prod_{\alpha\in\Sigma^{+}}
	\Big( 
	\frac{\langle\alpha,x^{+}\rangle}
	{1+\langle\alpha,x^{+}\rangle}
	\Big)^{m_{\alpha}}\,
	e^{2\langle\rho,x^{+}\rangle}
	\qquad\forall\,x^{+}\in\overline{\mathfrak{a}^{+}}. 
\end{align}
Here $K$ is equipped with its normalized Haar measure,
{$\mathbb{M}$} denotes the centralizer of $\exp\mathfrak{a}$ in $K$ and the volume 
of $K/{\mathbb{M}}$ can be computed explicitly, see \cite[Eq (2.2.4)]{AnJi1999}.
Recall that $\rho\in\mathfrak{a}^{+}$ denotes the half sum of all positive roots 
$\alpha \in \Sigma^{+}$ counted with their multiplicities $m_{\alpha}$:
\begin{align*}
	\rho\,=\,
	\frac{1}{2}\,\sum_{\alpha\in\Sigma^{+}} \,m_{\alpha}\,\alpha.
\end{align*}
We also define 
\[
\rho_{\min}: =\min_{x^{+}\in\mathfrak{a}^{+}, \; \|x^{+}\|=1}  \langle \rho, x^{+} \rangle \in(0, \|\rho\|].
\]
Finally, we say that a vector $H\in\mathfrak{a}$ lies on a wall if there exists a
root $\alpha\in\Sigma$ such that $\langle{\alpha,H}\rangle=0$. 
Otherwise, we say that $H$ stays away from the walls.

\subsection{Spherical Fourier analysis}
Let $\mathcal{S}(K \backslash{G}/K)$ be the Schwartz space of bi-$K$-invariant
functions on $G$. The spherical Fourier transform (Harish-Chandra transform)
$\mathcal{H}$ is defined by
\begin{align}\label{S2 HC transform}
	\mathcal{H}f(\lambda)\,
	=\,\int_{G}\diff{x}\,\varphi_{-\lambda}(x)\,u_0(x) 
	\qquad\forall\,\lambda\in\mathfrak{a},\
	\forall\,f\in\mathcal{S}(K\backslash{G/K}),
\end{align}
where $\varphi_{\lambda}\in\mathcal{C}^{\infty}(K\backslash{G/K})$ is the
spherical function of index $\lambda \in \mathfrak{a}$.
Denote by $\mathcal{S}(\mathfrak{a})^{W}$ the subspace 
of $W$-invariant functions in the Schwartz space $\mathcal{S}(\mathfrak{a})$. 
Then $\mathcal{H}$ is an isomorphism between $\mathcal{S}(K\backslash{G/K})$ 
and $\mathcal{S}(\mathfrak{a})^{W}$. The inverse spherical Fourier transform 
is given by
\begin{align}\label{S2 Inverse formula}
	u_0(x)\,
	=\,C_0\,\int_{\mathfrak{a}}\,\diff{\lambda}\,
	|\mathbf{c(\lambda)}|^{-2}\,
	\varphi_{\lambda}(x)\,
	\mathcal{H}f(\lambda) 
	\qquad\forall\,x\in{G},\
	\forall\,f\in\mathcal{S}(\mathfrak{a})^{W},
\end{align}
where the constant $C_0=2^{n-\ell}/(2\pi)^{\ell}|K/{\mathbb{M}}||W|$ depends only 
on the geometry of $\mathbb{X}$, and $|\mathbf{c(\lambda)}|^{-2}$ is the 
so-called Plancherel density. We next review some elementary facts about
the Plancherel density and the elementary spherical functions.

\subsubsection{Plancherel density}
According to the Gindikin-Karpelevič formula for the Harish-Chandra
$\mathbf{c}$-function, we can write the Plancherel density as
\begin{align*}
	|\mathbf{c}(\lambda)|^{-2}\,
	=\,\prod_{\alpha\in\Sigma_{r}^{+}}\,
	\Big|\mathbf{c}_{\alpha}
	\Big(
	\frac{\langle\alpha,\lambda\rangle}{\langle\alpha,\alpha\rangle}
	\Big)\Big|^{-2}
\end{align*}
with
\begin{align*}
	\mathbf{c}_{\alpha}(z)\,=\,
	\frac{\Gamma(\frac{\langle{\alpha,\rho}\rangle}
		{\langle{\alpha,\alpha}\rangle}
		+\frac{1}{2} m_{\alpha})}
	{\Gamma(\frac{\langle{\alpha,\rho}\rangle}
		{\langle{\alpha,\alpha}\rangle})}\,
	\frac{\Gamma(\frac{1}{2}
		\frac{\langle{\alpha,\rho}\rangle}
		{\langle{\alpha,\alpha}\rangle} 
		+\frac{1}{4} m _{\alpha} 
		+ \frac{1}{2} m_{2\alpha})}
	{\Gamma(\frac{1}{2}
		\frac{\langle{\alpha,\rho}\rangle}
		{\langle {\alpha,\alpha}\rangle} 
		+ \frac{1}{4} m_{\alpha})}\,
	\frac{\Gamma(iz)}
	{\Gamma(iz+ \frac{1}{2}m_{\alpha})}\,
	\frac{\Gamma(\frac{i}{2}z 
		+ \frac{1}{4} m_{\alpha})}
	{\Gamma(\frac{i}{2}z 
		+\frac{1}{4} m_{\alpha} 
		+ \frac{1}{2} m_{2\alpha})}.
\end{align*}
It holds
\begin{align*}
	\mathbf{b}(\lambda)\,
	=\,\bm{\pi}(i\lambda)\,\mathbf{c}(\lambda)\,
\end{align*}
where $\bm{\pi}(i\lambda)=\prod_{\alpha\in\Sigma_{r}^{+}}
\langle{\alpha,\lambda}\rangle$, and
$\mathbf{b}(-\lambda)^{\pm1}$ is a holomorphic function for 
$\lambda\in\mathfrak{a}+i\overline{\mathfrak{a}^{+}}$, which has the following
behavior:
\begin{align}\label{S2 bfunction}
	|\mathbf{b}(-\lambda)|^{\pm1}\,
	\asymp\,
	\prod_{\alpha\in\Sigma_{r}^{+}}\,
	(1+|\langle{\alpha,\lambda}\rangle|)^{
		\mp\frac{m_{\alpha}+m_{2\alpha}}{2}\pm1}.
\end{align}
Its derivatives can be estimated by
\begin{align}\label{S2 bfunction derivative}
	p(\tfrac{\partial}{\partial\lambda})
	\mathbf{b}(-\lambda)^{\pm1}\,
	=\,\mathrm{O}\big(|\mathbf{b}(-\lambda)|^{\pm1}\big),
\end{align}
where $p(\tfrac{\partial}{\partial\lambda})$ is any differential 
polynomial.

The following result will be needed later.
\begin{lemma}\label{lemma bfunc ratio}
	Let $x, y\in G$ such that $|x^{+}|=\textrm{O}(t)$ and $y$ is bounded. Then, as $t\rightarrow \infty$, it holds
	$$\frac{\textbf{b}\left(-i\frac{(y^{-1}x)^{+}}{2t}\right)^{-1}}{
		\textbf{b}\left(-i\frac{x^{+}}{2t}\right)^{-1}}=1 + \textnormal{O}\big(t^{-1}\big).$$
\end{lemma}
\begin{proof}
	Observe first that owing to \eqref{S2 bfunction}, for all $x\in G$ such that $|x^{+}|=\textrm{O}(t)$, we have
	\begin{equation}\label{bfunc comp}
		\textbf{b}\left(-i\frac{x^{+}}{2t}\right)^{-1}\asymp 1.
	\end{equation} Since $y$ is bounded, on the one hand we have $|(y^{-1}x)^{+}|=d(xK, yK)=\textrm{O}(t)$ by the triangle inequality, and on the other hand by the mean value theorem we get
	\begin{align*} 
		\left|	\textbf{b}\left(-i\frac{x^{+}}{2t}\right)^{-1}-\textbf{b}\left(-i\frac{(y^{-1}x)^{+}}{2t}\right)^{-1} \right| \lesssim \left|\frac{x^{+}}{2t}-\frac{(y^{-1}x)^{+}}{2t}\right|\lesssim \frac{|y|}{t} \lesssim t^{-1}.
	\end{align*} 
	Here, we used the derivative bound \eqref{S2 bfunction derivative} combined with \eqref{bfunc comp}, and the fact that
	$$|(y^{-1}x)^{+}-x^{+}|\leq |y|,$$
	by \eqref{S2 Distance}. Dividing with 
	$\textbf{b}\left(-i\frac{x^{+}}{2t}\right)^{-1}$, we get the desired result.
\end{proof}

\subsubsection{Spherical functions}
For every $\lambda\in\mathfrak{a}_{\mathbb{C}}$, the spherical function
$\varphi_{\lambda}$ is a smooth bi-$K$-invariant eigenfunction of all 
$G$-invariant differential operators on $\mathbb{X}$, in particular of the
Laplace-Beltrami operator:
\begin{equation*}
	-\Delta\varphi_{\lambda}(x)\,
	=\,(\langle \lambda, \lambda \rangle+|\rho|^2)\,\varphi_{\lambda}(x).
\end{equation*}
It is symmetric in the sense that
$\varphi_{\lambda}(x^{-1})=\varphi_{-\lambda}(x)$, $W$-invariant in $\lambda \in \mathfrak{a}_{\mathbb{C}}$,
and it is given by the integral representation
\begin{align}\label{S2 Spherical Function}
	\varphi_{\lambda}(x)\, 
	=\,\int_{K}\diff{k}\,e^{\langle{i\lambda+\rho,\,A(kx)}\rangle}.
\end{align} 
Furthermore, it holds
\begin{align}\label{S2 spherical split}
	\varphi_{\lambda}(y^{-1}x)\, 
	=\,\int_{K}\diff{k}\,e^{\langle{i\lambda+\rho,\,A(kx)}\rangle}\,e^{\langle{-i\lambda+\rho,\,A(ky)}\rangle},
\end{align}
see for instance \cite[Ch. III, $\S 1$, Theorem 1.1]{Hel1994}. 
Recall that all elementary spherical functions $\varphi_{\lambda}$ 
with parameter $\lambda\in\mathfrak{a}$ are controlled by the ground spherical
function $\varphi_{0}$, which satisfies the global estimate
\begin{align}\label{S2 global estimate phi0}
	\varphi_{0}(\exp{H})\,
	\asymp\,
	\Big\lbrace \prod_{\alpha\in\Sigma_{r}^{+}} 
	1+\langle\alpha,H\rangle\Big\rbrace\,
	e^{-\langle\rho, H\rangle}
	\qquad\forall\,H\in\overline{\mathfrak{a}^{+}}.
\end{align}
Moreover, the ground spherical function satisfies 
\begin{align}\label{S2 phi0 far}
	\varphi_{0}(\exp{H})\,
	\sim\,C_{2}\,\bm{\pi}(H)\,e^{-\langle{\rho,H}\rangle}
\end{align}
as $\mu(H)\rightarrow\infty$, where $C_{2}=\bm{\pi}(\rho_{0})^{-1}\mathbf{b}(0)$,
see for instance \cite[Proposition 2.2.12.(ii)]{AnJi1999}. Finally, for all $q>0$ we have
$$0<\varphi_{iq\rho}(\exp H)=\varphi_{-iq\rho}(\exp H)\leq e^{q\langle \rho, H\rangle}\varphi_{0}(\exp H),$$
see \cite[Proposition 4.6.1]{GaVa1988}. Then it becomes clear that $\varphi_{\pm i\rho}\equiv 1$ and that $f\in L^1(\mathbb{X})$ implies  $\int_{\mathbb{X}}\textrm{d}x \, |f(x)|\, \varphi_{\pm iq\rho}(x)<+\infty$ for all $q\in(0,1)$. 
For further properties of spherical functions 
$\varphi_{\lambda}$ we refer to \cite[Chap.4]{GaVa1988} and 
\cite[Chap.IV]{Hel2000}.

\subsection{Heat kernel on symmetric spaces}
The heat kernel on $\mathbb{X}$ is a positive bi-$K$-invariant right 
convolution kernel, i.e., $h_{t}(xK,yK)=h_{t}(y^{-1}x)>0$, 
which is thus determined by its restriction 
to the positive Weyl chamber. 
According to the inversion formula of the spherical Fourier transform, 
the heat kernel is given by
\begin{align}\label{S2 heat kernel inv}
	h_{t}(x)\,
	=\,C_{0}\,\int_{\mathfrak{a}}\,\diff{\lambda}\,
	|\mathbf{c(\lambda)}|^{-2}\,
	\varphi_{\lambda}(x)\,
	e^{-t(|\lambda|^{2}+|\rho|^{2})}
\end{align}
and satisfies the global estimate
\begin{align}\label{S2 heat kernel}
	h_{t}(x)\,
	\asymp\,t^{-\frac{n}{2}}\,
	\Big\lbrace{
		\prod_{\alpha\in\Sigma_{r}^{+}}
		(1+t+\langle{\alpha,x^{+}}\rangle)^{\frac{m_{\alpha}+m_{2\alpha}}{2}-1}
	}\Big\rbrace\,\varphi_{0}(x)
	e^{-|\rho|^{2}t-\frac{|x|^{2}}{4t}}
\end{align}
for all $t>0$ and $x^{+}\in\overline{\mathfrak{a}^{+}}$, 
see \cite{AnJi1999,AnOs2003}.

Let $x\in G$. If 
$\mu(x^{+})=\min_{\alpha\in\Sigma^{+}}\langle{\alpha,x^{+}}\rangle\rightarrow\infty$, so in this sense $x^{+}$ stays away from walls, or if $|x|=\textrm{O}(t)$, then the large time
behavior of the heat kernel $h_t(x)$ can be described more accurately by the following
asymptotics \cite[Theorem 5.1.1]{AnJi1999}:
\begin{align}\label{S2 heat kernel critical region}
	h_t(\exp x^{+})\,
	\sim\,C_{1}\,t^{-\frac{\nu}{2}}\,
	\mathbf{b}\big(-i\tfrac{x^{+}}{2t}\big)^{-1}\,
	\varphi_{0}(\exp{x^{+}})\,e^{-|\rho|^{2}t-\frac{|x^{+}|^{2}}{4t}}
\end{align}
as $t\rightarrow\infty$.
Here $C_{1}=C_{0}2^{-|\Sigma_{r}^{+}|}|W|\pi^{\frac{\ell}{2}}
\bm{\pi}(\rho_{0})\mathbf{b}(0)^{-1}$ is a positive constant, and $\rho_{0}$
denotes the half sum of all positive reduced roots.

\section{Critical regions and mass functions}\label{Section.3 lp critical}

In this section we discuss the notion $L^p$ concentration of the heat kernel, which will be a crucial concept for our results.

To begin with, it is well-known that the behavior of the $L^p$ norms $\|h_t\|_p$ of the heat kernel $h_t$ can be deduced from the pointwise estimates in \eqref{S2 heat kernel}. More precisely, for $t$ large, it was proved in \cite[Proposition 4.1]{AnJi1999} that
$$
\left\|h_t\right\|_p \asymp 
\begin{cases}t^{-\frac{\ell}{2 p^{\prime}}} e^{-\frac{4}{p p^{\prime}}|\rho|^2 t} & \text { if } 1 \leq p<2 \\ t^{-\frac{\ell}{4}-\frac{| \Sigma_r^{+} |}{2}} e^{-|\rho|^2 t} & \text { if } p=2, \\ t^{-\frac{\ell}{2}-|\Sigma_r^{+}|} e^{-|\rho|^2 t} & \text { if } 2<p \leq+\infty .
\end{cases}
$$

Davies \cite[Corollary 5.7.3]{Dav1989}, observed that heat diffusion on real hyperbolic spaces propagates asymptotically with finite speed. This striking phenomenon was later extended to non-compact symmetric spaces, by
analytic means in \cite{AnSe1992} and probabilistic ones in \cite{B94}. The general $L^p$ setting was considered in \cite{AnJi1999}. More precisely, the following lemma shows where the heat kernel
$h_{t}$ concentrates for any $1\leq p \leq \infty$. We denote by $(\widehat{H, \rho})$ the angle between $H\in \mathfrak{a}$ and $\rho$.

\begin{theorem}\cite[Theorem 4.1.2.]{AnJi1999} \label{thm: lp critical}
	(i) Case $1 \leq p<2$ : For $t$ large, consider the box
	$$
	B_p(t)=\left\{H \in \overline{\mathfrak{a}^{+}}: \quad 2| \rho_p|t-r(t) \leq| H|\leq 2| \rho_p | t+r(t),\quad  (\widehat{H, \rho})\leq \theta(t)\right\}
	$$
	centered at the point $2 t \rho_p := 4\left(\frac{1}{p}-\frac{1}{2}\right)t\rho$ in $\mathfrak{a}$, where $r(t)$ and $\theta(t)$ are positive with
	$$
	\left\{\begin{array} { l } 
		{ \frac { r ( t ) } { t } \rightarrow 0 } \\
		{ \frac { r ( t ) } { \sqrt { t } } \rightarrow + \infty }
	\end{array} \quad \text { and } \quad \left\{\begin{array}{l}
		\theta(t) \rightarrow 0 \\
		\sqrt{t} \theta(t) \rightarrow+\infty
	\end{array} \quad \text { as } t \rightarrow+\infty .\right.\right.
	$$
	Then the $L^p$ norm of $h_t$ concentrates asymptotically in the bi-K-orbit of $\exp B_p(t)$ :
	$$
	\left\|h_t\right\|_p^{-1}\left\{\int_{G \backslash K\left\{\exp B_p(t)\right\} K} d x\left|h_t(x)\right|^p\right\}^{1 / p} \longrightarrow 0 \quad \text { as } t \rightarrow+\infty .
	$$
	(ii) Case $p=2$ : We have a similar concentration
	$$
	\left\|h_t\right\|_2^{-1}\left\{\int_{G \backslash K\left\{\exp B_2(t)\right\} K} d x\left|h_t(x)\right|^2\right\}^{1 / 2} \longrightarrow 0
	$$
	with respect to the box
	$$
	B_2(t)=\left\{H \in \overline{\mathfrak{a}^{+}}:\quad r_1(t) \leq|H|  \leq r_2(t), \quad \omega(H) \geq r_3(t)\right\},
	$$
	where
	$$
	\left\{\begin{array} { l } 
		{  r_1 ( t )  \rightarrow +\infty} \\
		{ \frac { r_1 ( t ) } { \sqrt { t } } \rightarrow 0 }
	\end{array}, \quad \frac{r_2(t)}{\sqrt{t}}\rightarrow  +\infty \quad \text { and } \quad \left\{\begin{array}{l}
		r_3(t) \rightarrow +\infty \\
		\frac{r_3(t)}{\sqrt{t}}  \rightarrow 0
	\end{array} \quad \text { as } t \rightarrow+\infty .\right.\right.
	$$
	\\
	(iii) Case $2<p<+\infty$:
	$$\left\|h_t\right\|_p^{-1}\left\{\int_{|x| \geq R(t)} d x\left|h_t(x)\right|^p\right\}^{1 / p} \longrightarrow 0
	$$
	if $\frac{R(t)}{\log t} \rightarrow+\infty$, $\frac{R(t)}{\sqrt{t}} \rightarrow 0$ as $t \rightarrow+\infty$.
	\\
	(iv) Case $p=+\infty$ : $h_t(x)$ reaches its maximum at $x=e$ and
	$$\left\|h_t\right\|_{\infty}^{-1}\sup_{|x|\geq R(t)}h_t(x) \longrightarrow 0
	$$
	if $\frac{R(t)}{\log t} \rightarrow+\infty$, $\frac{R(t)}{\sqrt{t}} \rightarrow 0$ as $t \rightarrow+\infty$.
\end{theorem}\label{thm: critical}
Notice that $p= 2$ is the only case which resembles the Euclidean setting.
We will be referring to the above domains of heat concentration as \textit{$L^p$ critical regions}. It is also worth mentioning that rates of convergence can be explicitly computed, and they will depend on the quantities in $t$ used to define the regions, by using the heat kernel estimates \eqref{S2 heat kernel}. For instance, in the $L^1(\mathbb{X})$ case, rates can be found in \cite{AnSe1992} or \cite[Lemma 2.1]{APZ2023},  but for the sake of brevity we will avoid doing so.

We now consider the following mass functions of initial data, whenever the integrals make sense: set
$$M_p(u_0)(g):=\int_G \diff{y}\,u_{0}(yK)\,e^{\frac{2}{p} \langle \rho , A(k(g)^{-1}y)\rangle}, \qquad 1\leq p<2$$ 
where $g=k(g)(\exp g^{+})k'(g)$ in the Cartan decomposition, while we set
$$M_p(u_0)(g):=\frac{(u_{0}*\varphi_{0})(g)}{\varphi_{0}(g)}, \qquad 2\leq p\leq \infty.$$
Clearly, these functions are right-$K$-invariant, so they descend to functions on $\mathbb{X}$.

\begin{remark}\label{remark: mass bdd weight}
	Consider the weights $w_p$ on $\mathbb{X}$ given by
	$${w}_p(gK)=e^{\frac{2}{p}\langle \rho, g^{+} \rangle},\quad k\in K, \quad  \text{where} \quad g=k(g)(\exp g^{+}) k'(g), \quad 1\leq p <2,$$
	and 
	$${w}_p(gK)=e^{\langle \rho, g^{+} \rangle}, \quad 2\leq p\leq \infty,$$ 
	and recall that by \eqref{S2 ineq Iwasawa Cartan} that
	$$\langle \rho, A(kg) \rangle \leq \langle \rho, (kg)^{+} \rangle=\langle \rho, g^{+} \rangle \quad \text{for all} \quad  k\in K.$$

	Assume first that $1\leq p<2$. Then it becomes clear that for initial data $u_0$ such that $$\int_{G}\textrm{d}g \, |u_0(gK)|\, e^{\frac{2}{p}\langle \rho, g^{+} \rangle}<+\infty,$$ the mass function $M_p(u_0)(\cdot)$ for $1\leq p <2$ is  bounded. 
	
	Let us now discuss the case $2\leq p\leq \infty$ and the case of initial data $u_0$ such that $$\int_{G}\textrm{d}g \, |u_0(gK)|\, e^{\langle\rho,g^{+} \rangle }<+\infty.$$
	Owing to \eqref{S2 spherical split}, we have the following local Harnack inequality:
	\begin{align}\label{Harnack}
		\varphi_{0}(y^{-1}g)\,
		=\,
		\int_{K}\diff{k}\,e^{\langle\rho,A(kg)\rangle}\,
		e^{\langle{\rho,A(ky)}\rangle}\,\leq e^{\langle \rho,y^{+} \rangle}\,\int_{K}\diff{k}\,e^{\langle{\rho,A(kg)\rangle}}\,\leq \, e^{\langle\rho,y^{+} \rangle}\, \varphi_{0}(g),
	\end{align}
	which holds for every 
	$y,g\in{G}$. Therefore, 
	\begin{align*}
		\left|\frac{(u_{0}*\varphi_{0})(g)}{\varphi_{0}(g)}\right|\leq \varphi_{0}(g)^{-1}\int_{G}\textrm{d}y\, |u_0(y)|\, \varphi_{0}(y^{-1}g)\leq \int_{G}\textrm{d}y\, |u_0(y)|\, e^{\langle\rho,y^{+} \rangle}.
	\end{align*}
	This establishes that the mass function $M_p(u_0)(\cdot)$, $2 \leq p \leq \infty$, is bounded. 
	
	Clearly, the mass functions $M_p(u_0)(\cdot )$ are continuous and bounded for all $p\in [1, \infty]$ if  ${u}_{0}\in\mathcal{C}_{c}(\mathbb{X})$, since $\langle\rho,y^{+} \rangle \leq |\rho||y|$.
\end{remark}

\begin{remark}
	It is remarkable the different behavior of mass functions depending on whether $p$ is above or below the critical point $p=2$.  This contrast is a consequence of the different $L^p$ heat concentration. However, it is worth mentioning that for $p=2$ a mass function can be defined both ways and we can get for both the $L^2$ result of \cref{S1 Main thm 2}. We refer for this to \cref{prop: L2 variant}, which emphasizes that the case $p=2$ is indeed critical. Let us elaborate.
	
	Let $u_{0}\in \mathcal{C}_c({\mathbb{X}})$. 
	Denote by $\mathbb{M}$ the the centralizer of $\exp{\mathfrak{a}}$ in $K$. 
	Recall that the Helgason-Fourier transform 
	\begin{align}\label{S3 Helgason}
		\mathcal{H}u_{0}(\lambda,k\mathbb{M})\,
		=\,
		\int_{G}\diff{g}\,
		u_{0}(gK)\,e^{\langle{-i\lambda+\rho,\,A(k^{-1}g)}\rangle}, 
	\end{align} 
	boils down to the transform \eqref{S2 HC transform}
	when $u_{0}$ is bi-$K$-invariant. 
	It follows that 
	$$M_p(u_0)(g)=\mathcal{H}u_0\left( i\left(\frac{2}{p}-1\right)\rho,k(g)\mathbb{M}\right), \quad 1\leq p \leq 2.$$
	Consider now on the one hand the spherical principal series, which are realized on the boundary $K/\mathbb{M}$ by
	\begin{equation*}
		\bigl[\pi_\lambda(y)\,\xi\bigr](kM)=
		\underbrace{e^{-\langle\rho+i\lambda,H(y^{-1}k)\rangle}}_{e^{\langle\rho+i\lambda,A(k^{-1}y)\rangle}}
		\xi(\kappa(y^{-1}k)\mathbb{M})\,,
	\end{equation*}
	where $g=\kappa(g)\,(\exp H(g))\,n(g)$ in the Iwasawa decomposition $G=K(\exp\mathfrak a)N$.
	Then $M_p(u_0)$ is equal to
	\begin{equation}\label{LpMass2}
		\bigl[\pi_{-i(2/p-1)\rho}(u_0) \, 1\bigr](k(g)\mathbb{M})
		=\int_Gdy\,u_0(y)\,\bigl[\pi_{-i(2/p-1)\rho}(y)\, 1\bigr](k(g)\mathbb{M})
	\end{equation}
	or to the Helgason-Fourier transform, as already mentioned, 
	evaluated at $\lambda=i(2/p-1)\rho$ and $k(g)\mathbb{M}$. On the other hand, notice that
	\begin{equation*}
		\varphi_{\pm i(2/p-1)\rho}(g)=\langle\pi_{\mp i(2/p-1)\rho}(x)\,1,1\rangle\,,
	\end{equation*}
	hence
	\begin{equation*}
		\tfrac{u_0\,*\,\varphi_{\pm i(2/p-1)\rho}(g)}{\varphi_{\pm i(2/p-1)\rho}(g)}
		=\tfrac{\langle\pi_{\mp i(2/p-1)\rho}(g)\,1,\pi_{\pm i(2/p-1)\rho}(u_0)1\rangle}{\langle\pi_{\mp i(2/p-1)\rho}(g)\,1,1\rangle}
	\end{equation*}
	is quite different from \eqref{LpMass2}, so the expression of the mass function for $1\leq p<2$ does not resemble its corresponding expression for $p>2$.
\end{remark}

\begin{remark}\label{remark bi-K}
	If $u_{0}$ is in $\mathcal{L}^{p}(G//K)$, then we notice that for every $p$ the mass function boils down to a constant. Indeed, dealing with $p\in [1,2]$ first, 
	observe that for any $k\in K$ we have
	\begin{align*}
		\int_G \diff{y}\,u_{0}(yK)\,e^{\frac{2}{p} \langle \rho , A(ky)\rangle}=\mathcal{H}{u}_0\left( \pm i (\frac{2}{p}-1)\rho \right)=\int_{G}\diff{y}\,\varphi_{(\frac{2}{p}-1)i\rho}u_{0}(yK).
	\end{align*}
	Notice that when $p=1$, we simply obtain  $\int_G \diff{y}\,u_{0}(yK)$, since $\varphi_{\pm i\rho}\equiv 1$. Thus we obtain the mass usually considered in the Euclidean case or in \cite{VazHyp,APZ2023}. Hence the mass function boils down to a constant, namely to
	\begin{align*}
		{M}_p\,=\,\mathcal{H}u_{0}(\pm(2/p-1)i\rho) \quad \text{if} \quad f\in \mathcal{L}^p(G//K), \quad p\in [1, 2].
	\end{align*}
	Similarly, when $p\in [2, \infty]$,
	\begin{align}\label{S4 radial mass}
		\frac{(u_{0}*\varphi_{0})(g)}{\varphi_{0}(g)}\,
		&=\,\tfrac{1}{\varphi_{0}(g)}\,
		\int_{G}\diff{y}\,u_{0}(y)\varphi_{0}(y^{-1}g)\notag\\[5pt]
		&=\,\tfrac{1}{\varphi_{0}(g)}\,
		\int_{G}\diff{y}\,u_{0}(y)\,
		\underbrace{\vphantom{\Big|}
			\int_{K}\diff{k}\,\varphi_{0}(y^{-1}kg)
		}_{\varphi_{0}(g)\,\varphi_{0}(y)}\,
		=\,\int_{G}\diff{y}\,u_{0}(y)\,\varphi_{0}(y).
	\end{align}
	Hence the mass function boils down to a constant, namely to
	\begin{align*}
		{M}_p\,=\,\mathcal{H}u_{0}(0), \quad \text{if} \quad f\in \mathcal{L}^p(G//K), \quad p\in [2, \infty].
	\end{align*}
	In this sense, our notion of mass generalizes all constant masses considered before in \cite{VazHyp, APZ2023, NRS24}.
\end{remark}

\section{$L^p$ convergence, $1\leq p<2$} \label{Section.4 lp 1 to 2}

Our aim in this section is to prove the $L^p$ asymptotics \eqref{S1 Main thm 2 convergence} for $p\in [1,2)$. To this end, we examine separately the $L^p(\mathbb{X})$ critical region, as defined in \cref{thm: lp critical}, for which we rely on heat kernel asymptotics. Outside this critical region, we treat in a different way the $L^p$ concentration of the solution, the reason being that everything boils down to the small $L^p$ mass of the heat kernel itself there.

In this section we will assume that with the notation of \cref{thm: lp critical}, the angle $\theta(t)$ in the description of the $L^p(\mathbb{X})$ critical region for $1\leq p<2$, is equal to $\frac{r(t)}{t}$.

\begin{proposition}\label{S3 proposition kernel quotient}
	Let	$1\leq p<2$	and $y\in{G}\smallsetminus{K}$, $|y|<\xi$. 
	Then, for every $g$ in the critical region $K(\exp B_p(t))K$  we have
	\begin{align*}
		\frac{h_{t}(gK,yK)}{h_{t}(gK,eK)}\,
		=\,e^{\frac{2}{p}\langle{\rho,\,A(k(g)^{-1}y)}\rangle}
		+\textnormal{O}\big(\tfrac{r(t)}{t}\big)
		\qquad\textnormal{as}\quad\,t\rightarrow\infty,
	\end{align*}
	where $g=k(g)(\exp{g^{+}})k'(g)$ in the Cartan decomposition.
\end{proposition}

For the proof of \cref{S3 proposition kernel quotient} we need the following two lemmas, which describe the effect of a small translation on the critical region $K\{\exp B_p(t)\}K$, for $1\leq p \leq 2$. For our purposes, a modification of certain arguments in \cite{APZ2023} for $1\leq p<2$ would have been sufficient due the similar nature of the corresponding critical regions (see \cref{thm: lp critical}(i)). However, we have different propagation behaviors for $1\leq p <2$ and for the critical point $p=2$ which we also want to consider later. Moreover, since the arguments in \cite{APZ2023} were directly related to the heat equation, we prefer to present the following two results which are rather general and might be useful to other contexts. Recall that $(\widehat{H, \rho})$ denotes the angle between $H\in \mathfrak{a}$ and $\rho$.

\begin{lemma}\label{S3 lemma distance behaviors}
	Define in $\mathfrak{a}$ the set 
	$$\Theta_{t}=\{\, x^{+}\in \overline{\mathfrak{a}^{+}}: \quad |x^{+}|\geq R(t), \quad (\widehat{x^{+}, \rho})\leq \theta(t)\,\},$$
	where $R(t)$ and $\theta(t)$ are positive and monotonic, such that
	$$R(t)\rightarrow +\infty, \qquad \theta(t)\rightarrow 0, \qquad R(t)^{-1}= \textrm{O}(\theta(t)^2). $$ 
	Let $x\in K(\exp \Theta_{t})K $ and let $y$ be bounded. Then as $t\rightarrow \infty$,
	\begin{enumerate}[label=(\roman*)]
		\vspace{5pt}\item 
		$\frac{|(y^{-1}x)^{+}|}{|x^{+}|}, \frac{|x^{+}|}{|(y^{-1}x)^{+}|}$
		are both equal to $1+\textnormal{O}\big(R(t)^{-1}\big)$.
		
		\vspace{5pt}\item         
		$\frac{x^{+}}{|x^{+}|}$ and $\frac{(y^{-1}x)^{+}}{|(y^{-1}x)^{+}|}$
		are both equal to
		$\frac{\rho}{|\rho|}+\textnormal{O}\big(\theta(t)\big)$. 
		
		\vspace{5pt}\item
		For every $\alpha\in\Sigma^{+}$,
		$\frac{\langle{\alpha,(y^{-1}x)^{+}}\rangle}{
			\langle{\alpha,x^{+}}\rangle}
		=1+\textnormal{O}\big(\theta(t)\big)$.
		
		\vspace{5pt}\item
		$d(xK,eK)-d(xK,yK)
		=\langle{\frac{\rho}{|\rho|},A(k(x)^{-1}y)}\rangle
		+\textnormal{O}\big(\theta(t)\big)$, \quad $x=k(x)(\exp x^{+}) k'(x)$.
	\end{enumerate}
\end{lemma}

\begin{proof}
	The proof is an adaptation of \cite[Lemma 3.7]{P2024}. Assume that $d(xK, eK)\geq R(t)$ and $d(yK, eK)\leq \xi$, which implies by the triangle inequality that $d(xK, yK)\geq \frac{1}{2}\,R(t)$, for $t$ large enough.
	
	We deduce first $(i)$ by using
	\begin{align*}
		\frac{|(y^{-1}x)^{+}|}{|x^{+}|}\,
		= \frac{d(xK,yK)}{d(xK, eK)}
		=\,1+\textrm{O}\big(R(t)^{-1}\big).  
	\end{align*}
	The second assertion follows similarly.
	
	Next, for (ii), since $\cos(\widehat{x^{+}, \rho})=1+\textrm{O}(\theta(t)^2)$, we first have
	\begin{align}\label{eq: eq*}
		\left|\frac{x^{+}}{|x^{+}|}-\frac{\rho}{|\rho|} \right|^2
		=2\left(1-\langle  \frac{\rho}{|\rho|},\frac{x^{+}}{|x^{+}|}  \rangle \right)=\textrm{O}(\theta(t)^2),
	\end{align}
	while for the second asymptotics in (ii), we work similarly, observing that  $(y^{-1}x)^{+}=x^{+}+\textrm{O}(1)$ implies
	\begin{align}\label{star}
		\langle \frac{\rho}{|\rho|}, \frac{(y^{-1}x)^{+}}{|(y^{-1}x)^{+}|} \rangle & =\frac{|x^{+}|}{|(y^{-1}x)^{+}|} \langle  \frac{\rho}{|\rho|},\frac{x^{+}}{|x^{+}|}  \rangle +\textrm{O}(|x|^{-1}) \notag \\
		&= \langle  \frac{\rho}{|\rho|},\frac{x^{+}}{|x^{+}|}  \rangle +\textrm{O}(R(t)^{-1}) \notag\\
		&=1+\textrm{O}(\theta(t)^2), 
	\end{align}
	using (i), \eqref{eq: eq*} and the fact that  $R(t)^{-1}=\textrm{O}(\theta(t)^2)$.
	
	Let us next deduce (iii) from $(i)$ and $(ii)$. 
	For every positive root $\alpha$,
	\begin{align*}
		\frac{\langle{\alpha,(y^{-1}x)^{+}}\rangle}{\langle{\alpha,x^{+}}\rangle}\,
		&=\,
		\frac{\langle{\alpha,\frac{(y^{-1}x)^{+}}{|(y^{-1}x)^{+}|}}\rangle}{
			\langle{\alpha,\frac{x^{+}}{|x^{+}|}}\rangle}\,
		\frac{|(y^{-1}x)^{+}|}{|x^{+}|}\notag\\[5pt]
		&=\, 
		\frac{\langle{\alpha,\frac{\rho}{|\rho|}}\rangle
			+\textrm{O}\big( \theta(t)\big)}{
			\langle{\alpha,\frac{\rho}{|\rho|}}\rangle
			+\textrm{O}\big( \theta(t)\big)}\,
		\Big\lbrace{1+\textrm{O}\big(R(t)^{-1}\big)}\Big\rbrace\,
		=\,1+\textrm{O}(\theta(t)),
	\end{align*}
	by the assumption on $R(t)$ and the fact that $\theta(t)$ decreases to $0$.
	
	It remains to prove (iv). For that, we follow \cite[Lemma 3.8]{APZ2023}. To lighten notation, write $x=k(\exp{x^{+}})k'$
	in the Cartan decomposition and consider the Iwasawa decomposition
	$k^{-1}y=n(k^{-1}y)(\exp{A(k^{-1}y)})k''$ for some $k''\in{K}$. Then
	\begin{align*}\label{S3 distance decomposition in lemma}
		d(xK,yK)\,
		&=\,d\big(k(\exp{x^{+}})K,kn(k^{-1}y)(\exp{A(k^{-1}y)})K\big)
		\notag\\[5pt]
		&=\,d\big(\exp{(-x^{+})}[n(k^{-1}y)]^{-1}(\exp{x^{+}})K,
		\exp{(A(k^{-1}y)}-x^{+})K\big).
	\end{align*}
	and we write
	\begin{align*}
		d(xK,eK)-d(xK,yK)\,
		&=\,\overbrace{\vphantom{\Big|}
			d(xK,eK)-d\big(\exp{(A(k^{-1}y)}-x^{+})K,eK\big)}^{I}\\
		&+\,\underbrace{\vphantom{\Big|}
			d\big(\exp{(A(k^{-1}y)}-x^{+})K,eK\big)-d(xK,yK)}_{II}.
	\end{align*}
	On the one hand, $|II|$ tends exponentially fast to $0$, see \cite{APZ2023}.  On the other hand, we have
	\begin{align*}
		I\,
		=\,|x^{+}|-|A(k^{-1}y)-x^{+}|\,
		&=\,\frac{2\langle{x^{+},A(k^{-1}y)}\rangle-|A(k^{-1}y)|^{2}}{
			|x^{+}|+|A(k^{-1}y)-x^{+}|}\\[5pt]
		&=\,\big\langle{\tfrac{x^{+}}{|x^{+}|},A(k^{-1}y)}\big\rangle\,
		+\textrm{O}\big(\tfrac{1}{|x^{+}|}\big)\\[5pt]
		&=\,\big\langle{\tfrac{\rho}{|\rho|},A(k^{-1}y)}\big\rangle\,+\textrm{O}\big( \theta(t)\big)
	\end{align*}
	by using $(ii)$, the fact that $\lbrace{A(k^{-1}y)\,|\,k\in{K}}\rbrace$ 
	is a compact subset of $\mathfrak{a}$ and that $|x^{+}|^{-1}\lesssim R(t)^{-1}\lesssim\theta(t)$. This concludes the proof.
\end{proof}

The second lemma is also an adaptation of  \cite[Lemma 3.8]{P2024}.

\begin{lemma}\label{inner prod translation}
	Define in $\mathfrak{a}$ the set 
	$$\Theta_{t}=\{\, x^{+}\in \overline{\mathfrak{a}^{+}}: \quad |x^{+}|\geq R(t), \quad (\widehat{x^{+}, \rho})\leq \theta(t)\,\},$$
	where $R(t)$ and $\theta(t)$ are positive and monotonic, such that
	$$R(t)\rightarrow +\infty, \qquad \theta(t)\rightarrow 0, \qquad R(t)^{-1}= \textrm{O}(\theta(t)^2). $$ 
	Let $x\in K(\exp \Theta_{t})K $ and let $y$ be bounded. Then as $t\rightarrow \infty$,
	\begin{equation*}
		\langle \rho, x^{+}\rangle -\langle \rho,(y^{-1}x)^{+}\rangle=|\rho||x^{+}|-|\rho||(y^{-1}x)^{+}|+\textrm{O}( \theta(t)).
	\end{equation*}
\end{lemma}	

\begin{proof}
	Since the rank one case is trivial, let us consider $\ell\geq 2$. Observe first that the claim follows by 
	\begin{align}\label{cos trans}
		\cos(\widehat{(y^{-1}x)^{+}, \rho})=  \cos(\widehat{x^{+}, \rho})+\textrm{O}(\theta(t)\,|x^{+}|^{-1}).
	\end{align}
	Indeed, by \eqref{cos trans} and taking into account that $\cos(\widehat{x^{+}, \rho})=1+\textrm{O}(\theta(t)^2)$, we get
	\begin{align*}
		\langle \rho, x^{+}\rangle -\langle \rho,(y^{-1}x)^{+}\rangle&= |\rho||x^{+}|\cos(\widehat{x^{+}, \rho})- |\rho||(y^{-1}g)^{+}|\cos(\widehat{(y^{-1}g)^{+}, \rho})\\
		&=|\rho|\,|x^{+}|\cos(\widehat{x^{+}, \rho})-|\rho|\,|(y^{-1}x)^{+}|\,(\cos(\widehat{x^{+}, \rho})+\textrm{O}(\theta(t)\,|x^{+}|^{-1})) \\
		&=(|\rho||x^{+}|-|\rho||(y^{-1}x)^{+}|)\cos(\widehat{x^{+}, \rho}) +\textrm{O}\left(\frac{|(y^{-1}x)^{+}|}{|x^{+}|}\theta(t)\right)\\
		&=|\rho||x^{+}|-|\rho||(y^{-1}x)^{+}| +\textrm{O}(\theta(t)),
	\end{align*}
	since $\frac{|(y^{-1}x)^{+}|}{|x^{+}|}=1+\textrm{O}(R(t)^{-1})$ according to \cref{S3 lemma distance behaviors}(i).
	
	Therefore, it remains to prove \eqref{cos trans}. We have $(y^{-1}x)^{+}=x^{+}+\textrm{O}(1)$ and $|(y^{-1}x)^{+}|=|x^{+}|+\textrm{O}(1)$, and also $\frac{x^{+}}{|x^{+}|}, \, \frac{(y^{-1}x)^{+}}{|(y^{-1}x)^{+}|}=\frac{\rho}{|\rho|}+\textrm{O}(\theta(t))$ according to \cref{S3 lemma distance behaviors}(ii), which in turn yields that $\frac{(y^{-1}x)^{+}}{|(y^{-1}x)^{+}|}=\frac{x^{+}}{|x^{+}|}+\textrm{O}(\theta(t))$. By \eqref{star}, we get $\cos\left( \widehat{ (y^{-1}x)^{+},\rho}\right)=1+\textrm{O}(\theta(t)^2)$,
	hence we conclude $(\widehat{(y^{-1}x)^{+}, \rho})=\textrm{O}(\theta(t))$.
	Next, recall the coordinates on $\mathfrak{a}$ with respect to the basis $\delta_1, ..., \delta_{\ell-1}, \rho/|\rho|$ introduced in \cref{Section.2 Prelim}, and write
	$$x^{+}=\left(\xi,\, \xi_{\ell}\right), \qquad (y^{-1}x)^{+}=\left(\zeta, \, \zeta_{\ell}\right).$$
	Since $\langle x^{+}, \rho \rangle=\xi_{\ell}\,|\rho|$, we get
	\begin{equation*}
		\xi_{\ell}=|x^{+}|\cos(\widehat{x^{+}, \rho}), \qquad |\xi|=|x^{+}|\sin(\widehat{x^{+}, \rho}). 
	\end{equation*} 
	Similarly, we have
	\begin{equation*}
		\zeta_{\ell}=|(y^{-1}x)^{+}|\cos(\widehat{(y^{-1}x)^{+}, \rho}), \qquad |\zeta|=|(y^{-1}x)^{+}|\sin(\widehat{(y^{-1}g)^{+}, \rho}).
	\end{equation*}
	Therefore
	\begin{equation}\label{xi coord}	
		\frac{|\xi|}{\xi_{\ell}}=\tan(\widehat{x^{+}, \rho})=\textrm{O}(\theta(t)), \qquad	|x^{+}|\asymp \xi_{\ell}
	\end{equation}
	and 
	\begin{equation}\label{zeta coord}	
		\frac{|\zeta|}{\zeta_{\ell}}=\tan(\widehat{(y^{-1}x)^{+}, \rho})=\textrm{O}(\theta(t)), \qquad	|(y^{-1}x)^{+}|\asymp \zeta_{\ell}.	
	\end{equation}
	Thus, we have
	\begin{align}\label{cos frac}
		\left| \cos(\widehat{(y^{-1}x)^{+}, \rho})- \cos(\widehat{x^{+}, \rho})  \right||x^{+}|&=\frac{\left|\zeta_{\ell}\,|x^{+}|-\xi_{\ell}\,|(y^{-1}x)^{+}|\right|}{|(y^{-1}x)^{+}|} \notag \\
		&\asymp \frac{\left| \zeta_{\ell}\sqrt{|\xi|^2+\xi_{\ell}^2}-\xi_{\ell}\sqrt{|\zeta|^2+\zeta_{\ell}^2}\right|}{\zeta_{\ell}} \notag \\
		&\asymp \frac{\left| \zeta_{\ell}^2\,|\xi|^2 - \xi_{\ell}^2\,|\zeta|^2\right|}{\zeta_{\ell}^2\,\xi_{\ell}\left( \sqrt{\left( \frac{|\xi|}{\xi_\ell}\right)^2+1} +\sqrt{\left( \frac{|\zeta|}{\zeta_\ell}\right)^2+1} \right)} \notag \\
		&\asymp \frac{\left| \zeta_{\ell}\,|\xi| - \xi_{\ell}\,|\zeta|\right| \left| \zeta_{\ell}\,|\xi| + \xi_{\ell}\,|\zeta|\right|}{\zeta_{\ell}^2\,\xi_{\ell}},
	\end{align}	
	due to \eqref{xi coord} and \eqref{zeta coord}. The fact that $(y^{-1}x)^{+}=x^{+}+\textrm{O}(1)$ implies $\zeta_{\ell}=\xi_{\ell}+\textrm{O}(1)$ and $|\zeta|=|\xi|+\textrm{O}(1)$, therefore
	\begin{equation}\label{cos frac asymp}
		\frac{ \zeta_{\ell}\,|\xi| - \xi_{\ell}\,|\zeta| }{\xi_{\ell}}=\textrm{O}(1), \qquad \frac{ \zeta_{\ell}\,|\xi| + \xi_{\ell}\,|\zeta|}{\zeta_{\ell}^2}=\textrm{O}\left(\frac{|\zeta|}{\zeta_{\ell}}\right)=\textrm{O}(\theta(t)).
	\end{equation}
	Altogether, we conclude  by \eqref{cos frac} and \eqref{cos frac asymp} that 
	$$ \cos(\widehat{(y^{-1}x)^{+}, \rho})- \cos(\widehat{x^{+}, \rho})= \textrm{O}(\theta(t)\,|x^{+}|^{-1}).  $$
\end{proof}

Now, let us turn to the proof of \cref{S3 proposition kernel quotient}.
\begin{proof}[Proof of \cref{S3 proposition kernel quotient}]
	First of all, notice that if $g\in K(\exp B_p(t))K$, $1\leq p<2$, then $g^{+}$ stays around the $\rho$-axis at angle $O(r(t)/t)$ and its distance to the origin is comparable to $t$. The same is true for $(y^{-1}g)^{+}$, as explained in the proof of \cref{inner prod translation} for $R(t)\asymp t$ and $\theta(t)=r(t)/t$. Thus \cref{S3 lemma distance behaviors} and \cref{inner prod translation} apply for $R(t)\asymp t$ and $\theta(t)=r(t)/t$.
	
	Thus, we may use the asymptotics \eqref{S2 heat kernel critical region} for the heat kernel, and the asymptotics \eqref{S2 phi0 far} for the ground spherical function to write
	\begin{align*}
		\frac{h_{t}(gK,yK)}{h_{t}(gK,eK)}\,
		&=\,\frac{h_{t}(\exp(y^{-1}g)^{+})}{h_{t}(\exp(g^{+}))}\\[5pt]
		&\sim\,
		\frac{\mathbf{b}\big(-i\tfrac{(y^{-1}g)^{+}}{2t}\big)^{-1}}{
			\mathbf{b}\big(-i\tfrac{g^{+}}{2t}\big)^{-1}}\,
		\frac{\bm{\pi}((y^{-1}g)^{+})}{\bm{\pi}(g^{+})}\,
		\frac{e^{-\langle{\rho,(y^{-1}g)^{+}}\rangle}}{
			e^{-\langle{\rho,g^{+}}\rangle}}\,
		\frac{e^{-\frac{|(y^{-1}g)^{+}|^{2}}{4t}}}{
			e^{-\frac{|g^{+}|^{2}}{4t}}}
	\end{align*}
	as $t\rightarrow\infty$. 
	The asymptotic behaviors of the first two factors are based on
	\cref{S3 lemma distance behaviors}(iii).
	On the one hand,
	\begin{align}\label{S3 counterexample 2}
		\frac{\bm{\pi}((y^{-1}g)^{+})}{\bm{\pi}(g^{+})}\,
		=\,\prod_{\alpha\in\Sigma_{r}^{+}}
		\frac{\langle{\alpha,(y^{-1}g)^{+}}\rangle}{\langle{\alpha,g^{+}}\rangle}\,
		=\,1+\textrm{O}\big(\tfrac{r(t)}{t}\big).
	\end{align}
	On the other hand, since $g^{+}, (y^{-1}g)^{+}=\textrm{O}(t)$, by using  \cref{lemma bfunc ratio} we have 
	\begin{align}\label{S3 counterexample 1}
		\frac{\mathbf{b}\big(-i\tfrac{(y^{-1}g)^{+}}{2t}\big)^{-1}}{
			\mathbf{b}\big(-i\tfrac{g^{+}}{2t}\big)^{-1}}\,=\,1+\textrm{O}\big(\tfrac{1}{t}\big).
	\end{align}
	For the Gaussian factors, we notice on the one hand that
	\begin{align}\label{S3 norm distance}
		-\frac{|(y^{-1}g)^{+}|^{2}}{4t}+\frac{|g^{+}|^{2}}{4t}\,
		&=\,\frac{|g^{+}|+|(y^{-1}g)^{+}|}{4t}\,
		\big(|g^{+}|-|(y^{-1}g)^{+}|\big)\notag\\[5pt]
		&=\,\underbrace{\vphantom{\Big|}
			\frac{d(gK,eK)+d(gK,yK)}{4t}
		}_{=\,\left( \frac{2}{p}-1\right)|\rho|+\textrm{O}\big(\tfrac{r(t)}{t}\big)}
		\underbrace{\vphantom{\frac{|g^{+}|+|(y^{-1}g)^{+}|}{4t}}
			\big\lbrace{d(gK,eK)-d(gK,yK)}\big\rbrace
		}_{=\,\langle{\frac{\rho}{|\rho|},A(k^{-1}y)}\rangle
			+\textrm{O}\big(\tfrac{r(t)}{t}\big)}\notag\\[5pt]
		&=\,\left(\frac{2}{p}-1\right)\langle{\rho,A(k(g)^{-1}y)}\rangle
		+\textrm{O}\big(\tfrac{r(t)}{t}\big)
	\end{align}
	according to \cref{S3 lemma distance behaviors}(iv) and the fact that $|A(k^{-1}y)|\leq |k^{-1}y|=|y|$, thus bounded. Therefore
	\begin{align}\label{S3 counterexample 3}
		e^{-\frac{|(y^{-1}g)^{+}|^{2}}{4t}+\frac{|g^{+}|^{2}}{4t}}\,
		=\,e^{\left(\frac{2}{p}-1\right)\langle{\rho,A(k(g)^{-1}y)}\rangle}+\textrm{O}\big(\tfrac{r(t)}{t}\big).
	\end{align}
	On the other hand, applying \cref{inner prod translation} for  $\theta(t)=\frac{r(t)}{t}$, we get
	\begin{align}\label{S3 counterexample 4}
		e^{-\langle{\rho,(y^{-1}g)^{+}}\rangle+\langle{\rho,g^{+}}\rangle}\,
		=\,e^{\langle{\rho,A(k(g)^{-1}y)}\rangle}+\textrm{O}\big(\tfrac{r(t)}{t}\big).
	\end{align}
	In conclusion, we deduce from \eqref{S3 counterexample 1}, 
	\eqref{S3 counterexample 2}, \eqref{S3 counterexample 3} and 
	\eqref{S3 counterexample 4} that 
	\begin{align*}
		\frac{h_{t}(gK,yK)}{h_{t}(gK,eK)}\,
		=\,e^{\frac{2}{p}\langle{\rho,\,A(k(g)^{-1}y)}\rangle}
		+\textrm{O}\big(\tfrac{r(t)}{t}\big)
	\end{align*}
	with $\frac{r(t)}{t}\rightarrow0$ as $t\rightarrow\infty$.
\end{proof}

We are now ready to prove the $L^p$ convergence of \cref{S1 Main thm 2}, $1\leq p<2$, for continuous and compactly supported initial data. Recall that any element of the group $G$ can be written as $g=k(g)(\exp g^{+})k'(g)$ in the Cartan decomposition. Assume that $\supp u_0\subseteq K\{\exp B(0,\xi)\}K$, for some $\xi>0$. We aim to study the difference
\begin{align}\label{S4 difference p1,2}
	u(t,g)-M_p(u_0)(g)\,h_{t}(g)\,
	&=(u_{0}*h_{t})(g)
	\,-\,
	{h}_{t}(g)\,
	\int_G \diff{y}\,u_{0}(yK)\,e^{\frac{2}{p} \langle \rho , A(k(g)^{-1}y)\rangle}
	\notag\\[5pt]
	&=\,{h}_{t}(g)\,
	\int_{G}\diff{y}\,u_{0}(yK)\,
	\Big\lbrace{
		\frac{h_{t}(y^{-1}g)}{h_{t}(g)}
		-e^{\frac{2}{p} \langle \rho , A(k(g)^{-1}y)\rangle}
	}\Big\rbrace.
\end{align}
According to the previous lemma, we have
\begin{align*}
	\frac{h_{t}(y^{-1}g)}{h_{t}(g)}
	-e^{\frac{2}{p} \langle \rho , A(k(g)^{-1}y)\rangle}\,
	=\,\mathrm{O}\Big(\frac{r(t)}{t}\Big)
	\qquad\forall\,g\in K\{\exp B_p(t)\}K, \, \, 1\leq p<2,  \, \, 
	\forall\,y\in \textrm{supp}u_0.
\end{align*}
Therefore over the critical region the $L^p$ integral of
${u}(t,\cdot)-M_p(u_0)\,{h}_{t}$, i.e. the quantity
\begin{align*}
	\left( \int_{K\left\{\exp B_p(t)\right\} K}
	\textrm{d}{g}\,    
	|u(t,g)-{M}_p(u_0)(g)\,{h}_{t}(g)|^p\, \right)^{1/p}
	\lesssim\,
	\frac{r(t)}{t}\,
	\underbrace{\vphantom{\Big|}
		\left(	\int_{G}\textrm{d}{g}\,{h}_{t}(g)^p \right)^{1/p}
	}_{\|h_t\|_p}\,\|u_0\|_{1},
\end{align*}
is $\textrm{o}(\|h_t\|_p)$.

It remains for us to check that the integral
\begin{align*}
	\left( \int_{G\smallsetminus K\left\{\exp B_p(t)\right\} K}\textrm{d}{g}\,
	|u(t,g)-M_p(u_0)(g)\,{h}_{t}(g)|^p\, \right)^{1/p}
	&\le\,
	\left( \int_{G\smallsetminus K\left\{\exp B_p(t)\right\} K}\textrm{d}{g}\,
	|u(t,g)|^p \right)^{1/p}\\[5pt]
	&+\,
	\left( \int_{G\smallsetminus K\left\{\exp B_p(t)\right\} K}\textrm{d}{g}\,    
	|{M}_p(u_0)(g)|^p \,h_{t}(g)^p  \right)^{1/p}
\end{align*}
is also $\textrm{o}(\|h_t\|_p)$. On the one hand, we know that ${M}_p(u_0)(\cdot)$ is bounded when $u_0$ is compactly supported,
and that the heat kernel ${h}_{t}$ asymptotically concentrates, in the $L^p$ sense, in 
$K\left\{\exp B_p(t)\right\} K$, hence
\begin{align*}
	\|h_t\|_p^{-1}	\left( \int_{G\smallsetminus K\left\{\exp B_p(t)\right\} K}\textrm{d}{g}\,
	|{M}_p(g)|^p \, {h}_{t}(g)^p\right)^{1/p}\,
	\longrightarrow\,0
\end{align*} 
as $t\rightarrow\infty$. Notice that the only important property of $M_p(u_0)$ at this point is simply the fact that it is bounded and not its expression per se. On the other hand, notice that for all
$y\in\supp{u_{0}}$ and for all $g\in{G}$ such that
$g^{+}\notin B_p(t)$, we have $(y^{-1}g)^{+}\,\notin\,{B_p(t)}'$, where
$$
B_p(t)'=\left\{H \in \mathfrak{a}: \quad 2| \rho_p|t-\frac{1}{2}r(t) \leq| H|\leq 2| \rho_p | t+\frac{1}{2}r(t),\quad  {(\widehat{H, \rho})}\leq \frac{r(t)}{2t}\right\}, \quad 1\leq p<2,$$
(see for instance the proof of \cite[Lemma 4.1]{P2024} or see \cite[p.19]{APZ2023} and use the fact that for the longest element $w_0$ of the Weyl group we have $(x^{-1})^{+}=-w_0.x^{+}$ for all $x\in G$ and $w_0.\rho=-\rho$).
Hence if $\supp u_0 \subseteq K\{\exp B(0, \xi)\}K$, by the Minkowski inequality we get
\begin{align*}
	\|f \ast h_t\|_{L^p(G\setminus K\{\exp B_p(t)\}K)}
	&= \left(\int_{G\setminus K\{\exp B_p(t)\}K} \textrm{d}g \left|\int_{K\{\exp B(0, \xi)\}K}\textrm{d}y \,u_0(y) h_t(y^{-1}g)\right|^p \right)^{1/p}\\ 
	&\leq  \int_{K\{\exp B(0, \xi)\}K} \textrm{d}y \,|u_0(y)|\left(\int_{G\setminus K\{\exp B_p(t)\}K} \textrm{d}g \, h_t(y^{-1}g)^p\right)^{1/p} \\
	&\leq  \int_{K\{\exp B(0, \xi)\}K} \textrm{d}y\, |u_0(y)|\left(\int_{G\setminus K\{\exp B_p(t)'\}K} \textrm{d}z \,h_t(z)^p\right)^{1/p},
\end{align*}
thus
$$\|u(t, \, \cdot \,)\|_{L^p(G\smallsetminus K\left\{\exp B_p(t)\right\}K)}=\textrm{o}(\|h_t\|_p).$$
This concludes the proof of the heat asymptotics in $L^{p}$, $1\leq p<2$, and for initial data
$u_{0}\in\mathcal{C}_{c}(\mathbb{X})$.

\subsection{$L^p$ heat asymptotics for other initial data, $1\leq p<2$}\label{Subsect other data 1 to2 2}

So far we have obtained above the long-time asymptotic convergence in $L^{p}$
($1\le{p}<2$) for the heat equation 
with compactly supported initial data, in the sense of \eqref{S1 Main thm 2 convergence}.
In this section we discuss convergence for further classes of initial data.

\begin{corollary}
	The $L^p$ asymptotic convergence \eqref{S1 Main thm 2 convergence}, $1\leq p<2$, still holds with initial data
	$u_{0}\in \mathcal{L}^p(G//K)$.
\end{corollary}

\begin{proof}
	The proof follows from the fact that mass functions boil down to constants (see \cref{remark bi-K}), a density argument and the Herz principle, see for instance \cite{APZ2023, NRS24}.
\end{proof}

\begin{corollary}\label{Cor Lp p12}
	The $L^p$ asymptotic convergence \eqref{S1 Main thm 2 convergence}, $1\leq p<2$, still holds without assuming bi-$K$-invariance for $u_0$ but under the assumption that ${u}_{0}\in L_{w_p}(\mathbb{X})$, that is, 
	\begin{align}\label{S4 other data assumption0}
		\int_{G}\diff{g}\,|u_{0}(gK)|e^{\frac{2}{p}\langle{\rho,g^{+}}\rangle}\,
		<\,\infty.
	\end{align}
\end{corollary}

\begin{proof}
	Recall first that under the assumption \eqref{S4 other data assumption0}, the mass function is bounded, as already mentioned in the \cref{remark: mass bdd weight}.
	For proving the $L^{p}$ convergence, 
	we argue by density.
	Since $u_{0}(gK)e^{\frac{2}{p}\langle{\rho,g^{+}}\rangle}$ belongs to $L^{1}(\mathbb{X})$,
	there exists a function $U_{0}(gK)e^{\frac{2}{p}\langle{\rho,g^{+}}\rangle}$ in
	$\mathcal{C}_{c}(\mathbb{X})$ such that 
	\begin{align*}
		\int_{G}\diff{g}\,
		|u_{0}(gK)-U_{0}(gK)|\,e^{\frac{2}{p}\langle{\rho,g^{+}}\rangle}\,
		<\,\tfrac{\varepsilon}{3}
	\end{align*}
	for every $\varepsilon>0$. So we clearly also have
	\begin{align*}
		\|u_0-U_0\|_1=	\int_{G}\diff{g}\,
		|u_{0}(gK)-U_{0}(gK)|<\,\tfrac{\varepsilon}{3}.
	\end{align*}
	Let $U={U}_{0}*{h}_{t}$ be the corresponding
	solution to the heat equation, and denote by 
	${M}_p(U_0)(g)=\int_G \diff y\, U_0(yK)\, e^{\frac{2}{p}\langle \rho, A(k(g)^{-1}y)\rangle}$ 
	the corresponding mass of $U_{0}$. On the one hand, there exists $T>0$
	such that for all $t>T$, we have
	\begin{align*}
		\|{U}(t,\,\cdot\,)
		-{M}_{p}(U_0)(\cdot)\,{h}_{t}\|_{L^{p}(\mathbb{X})}\,
		<\,\tfrac{\varepsilon}{3} \, \|h_t\|_p	
	\end{align*}
	since ${U}_{0}\in\mathcal{C}_{c}(\mathbb{X})$. 
	On the other hand, for every $g\in{G}$, we have
	\begin{align*}
		|M_p(u_0)(g)-M_p(U_0)(g)|\,&\leq
		\int_{G}\diff{y}\,|u_{0}(yK)-U_{0}(gK)|\,e^{\frac{2}{p}\langle \rho, A(k(g)^{-1}y)\rangle}\\[5pt]
		&\le\,\int_{G}\diff{y}\,|u_{0}(yK)-U_{0}(gK)|\,
		e^{\frac{2}{p}\langle{\rho,y^{+}}\rangle}\,
		<\,\tfrac{\varepsilon}{3}.
	\end{align*}
	We conclude by using the above relations and the triangle inequality:
	\begin{align*}
		\|u(t, \cdot)-M_p(u_0)(\cdot)\,h_t\|_{L^{p}(\mathbb{X})} &\leq \underbrace{\|u(t, \cdot)-U(t, \cdot)\|_p}_{\leq \|u_0-U_0\|_1\, \|h_t\|_p} + \|(M_p(U_0)(\cdot)-M_p(u_0)(\cdot))\, h_t\|_p \\
		& +\|U(t,\cdot)-M_p(U_0)(\cdot)\,h_t\|_p\\
		&< \frac{\varepsilon}{3}\, \|h_t\|_p+\frac{\varepsilon}{3}\, \|h_t\|_p+\frac{\varepsilon}{3}\, \|h_t\|_p \\
		&< \varepsilon\, \|h_t\|_p.
	\end{align*}
\end{proof}

Finally, the proposition that follows shows that except for $M_2(u_0)(g)=\frac{(u_0\ast\varphi_{0})(g)}{\varphi_{0}(g)}$ (which we will prove that it is a suitable mass function in the next section), also the function
$$\widetilde{M}_2(u_0)(g)=\int_G \textrm{d}y \, u_0(yK)\, e^{\langle \rho, A(k(g)^{-1}y) \rangle}$$
yields the $L^2$ convergence of \eqref{S1 Main thm 2 convergence}. In this sense, the index $p=2$ resembles both cases $p<2$ and $p>2$, so it is indeed a critical point.

\begin{proposition} \label{prop: L2 variant}
	The $L^2$ asymptotic convergence 
	\begin{align}\label{conv L2 variant}
		\|h_t\|_2^{-1}	\, \|u(t,\,\cdot\,)-\widetilde{M}_2(u_0)(\cdot)\,h_{t}\|_{L^{2}(\mathbb{X})}\,
		\longrightarrow\,0
		\qquad\textnormal{as}\quad\,t\rightarrow\infty
	\end{align}
	holds for $u_0\in \mathcal{C}_c(\mathbb{X})$, for $u_0\in \mathcal{L}^2(G//K)$ or for $u_{0}$ such that 
	\begin{align*}
		\int_{G}\diff{g}\,|u_{0}(gK)|e^{\langle{\rho,g^{+}}\rangle}\,
		<\,\infty.
	\end{align*}
\end{proposition}

\begin{proof}
	The claim for $u_0\in \mathcal{L}^2(G//K)$ follows immediately by the fact that for bi-$K$-invariant initial data the mass function $\widetilde{M}_2(u_0)$ boils down to $\mathcal{H}u_0(0)$, and \eqref{Ind convergence}. For the two remaining classes of initial data, it is clear by the arguments of Corollary \ref{Cor Lp p12} that it suffices to prove \eqref{conv L2 variant} for $u_0\in \mathcal{C}_c(\mathbb{X})$. In this case, since $\widetilde{M}_2(u_0)$ is bounded when $\supp u_0$ is contained in some \enquote{ball} $K\{\exp B(0, \xi)\}K$, standard arguments imply that the convergence outside the $L^2$ critical region will be true. Thus it remains to prove convergence in the critical region. 
	
	To this end, we sharpen our $L^2$ critical region defined in \cref{thm: lp critical}(ii). Let $\varepsilon(t)$ be a positive function decreasing to $0$, such that $\varepsilon(t)^{10}\sqrt{t}\rightarrow +\infty$. Then the set $K\exp\{\widetilde{B}_2(t)\}K$, where
	$$ \widetilde{B}_2(t)=\left\{g^{+} \in\overline{\mathfrak{a}^{+}}:\quad \varepsilon(t)\sqrt{t} \leq|g^{+}|  \leq \frac{\sqrt{t}}{\varepsilon(t)}, \quad (\widehat{g^{+}, \rho}) \leq \varepsilon(t) \right\},
	$$
	is an $L^2$ critical region for the heat kernel on $\mathbb{X}$ in the sense of \cref{thm: lp critical}(ii). Observe that if $g^{+}\in \widetilde{B}_2(t)$, then $\mu(g^{+})\gtrsim \varepsilon(t)\sqrt{t}$. Indeed, let us write in the $\mathfrak{a}$ coordinates of \cref{Section.2 Prelim}, as in \cref{inner prod translation}, 
	$$\alpha=(\beta, \beta_{\ell}), \quad \alpha\in \Sigma^{+}, \quad {\text{and}} \quad g^{+}=(\zeta, \zeta_{\ell}), \quad g^{+}\in \widetilde{B}_2(t),$$
	where $\beta_{\ell}, \zeta_{\ell}>0$. Then, $$\zeta_{\ell}\asymp |g^{+}|, \quad \frac{|\zeta|}{\zeta_{\ell}}=\textrm{O}(\varepsilon(t)).$$
	Therefore, for any $\alpha \in \Sigma^{+}$,
	$$\frac{\langle \alpha, g^{+} \rangle}{\zeta_{\ell}}=\langle \beta, \frac{\zeta}{\zeta_{\ell}} \rangle +\beta_{\ell} \geq \frac{1}{2} \,\beta_{\ell}$$
	for $t$ large enough, thus the claim follows by $\min_{\alpha\in\Sigma^{+}}\langle \alpha, g^{+} \rangle\geq \text{const}(\Sigma) \,\zeta_{\ell} \gtrsim |g^{+}|\gtrsim \varepsilon(t)\sqrt{t} $.

	It remains therefore to prove, as already mentioned, that
	$$\|h_t\|_2^{-1}	\, \|u(t,\,\cdot\,)-\widetilde{M}_2(u_0)(\cdot)\,h_{t}\|_{L^{2}(K\{\exp \widetilde{B}_2(t)\}K)}\,
	\longrightarrow\,0
	\qquad\textnormal{as}\quad\,t\rightarrow\infty.$$
	To this end, since
	\begin{align}\label{eq: diff heat L2 planB}
		u(t,g)\,-\,\widetilde{M}_2(u_0)(g)\,h_{t}(g)\,
		=\,h_{t}(g)\,
		\int_{K\exp\{B(0, \xi)\}K}\diff{y}\,u_{0}(yK)\,
		\Big\lbrace{
			\frac{h_{t}(y^{-1}g)}{h_{t}(g)}
			-e^{\langle \rho, A(k(g)^{-1}y) \rangle} \Big\rbrace},
	\end{align} 
	we see that the crucial ingredient is asymptotics of the heat kernel quotient $h_t(y^{-1}g)/h_t(g)$, when $g\in K\{\exp \widetilde{B}_2(t)\}K$ and $|y|<\xi$. On the one hand, it is proven in \cite[Lemma 4.7]{APZ2023} that 
	\begin{equation}\label{eq: heat quot L2 plan B}
		\frac{h_t(y^{-1}g)}{h_t(g)}-\frac{\varphi_{0}(y^{-1}g)}{\varphi_{0}(g)}=\frac{1}{\varepsilon(t)\sqrt{t}}, \qquad\forall\,g\in{K\{\exp \widetilde{B}_2(t)\}K},\,\,
		\forall\,|y|<\xi,
	\end{equation}  
	where $\varepsilon(t)\sqrt{t}\rightarrow +\infty$, $\varepsilon(t)\rightarrow 0$. On the other hand, one can obtain
	\begin{align*}
		g^{-1}y\,\in\,K\{\exp \widetilde{B}_2(t)'\}K
		\qquad\forall\,g\in{K\{\exp \widetilde{B}_2(t)\}K},\,\,
		\forall\,|y|<\xi,
	\end{align*}
	where
	\begin{align*}
		\widetilde{B}_2(t)'\,
		=\,\big\lbrace{
			H\in\overline{\mathfrak{a}^{+}}: \quad 
			\frac{1}{2}{\varepsilon}(t)\sqrt{t}\le|H|\le\tfrac{2\sqrt{t}}{\varepsilon(t)}
			\,\,\,\textnormal{and}\,\,\, (\widehat{H, \rho})=\textrm{O}(\varepsilon(t))
		}\big\rbrace.
	\end{align*}
	This follows from the proof of \cref{inner prod translation} concerning the effect of a small translation to the angle with the $\rho$ axis, while for the radii the claim follows simply by the triangle inequality. Thus the asymptotics \eqref{S2 phi0 far} for the ground spherical function hold true (notice that these are not available if $g$ lies in any $L^p$ critical region for $p>2$, since those contain the origin). In fact, according to \cite[Proposition 4.2]{APZ2023} the asymptotics for $\varphi_{0}$ can be sharpened in the present region:
	\begin{align*}
		\varphi_{0}(\exp{H})\,
		=\,\Big\lbrace{C_{3}+\textnormal{O}
			\big(\tfrac{1}{\mu(H)}}\big)\Big\rbrace\,
		\bm{\pi}(H)\,e^{-\langle{\rho,H}\rangle},
	\end{align*}
	for some explicit constant $C_3>0$. Then 
	$$ \frac{\varphi_{0}(y^{-1}g)}{\varphi_{0}(g)}\,
	=\underbrace{
		\tfrac{C_{3}+\mathrm{O}\big(\tfrac{1}{\varepsilon(t)\sqrt{t}}\big)}{
			C_{3}+\mathrm{O}\big(\tfrac{1}{\varepsilon(t)\sqrt{t}}\big)}
	}_{1+\mathrm{O}\big(\tfrac{1}{\varepsilon(t)\sqrt{t}}\big)}\, \prod_{\alpha\in\Sigma_{r}^{+}} \frac{\langle \alpha, (y^{-1}g)^{+} \rangle}{\langle \alpha, g^{+} \rangle} \, e^{-\langle{\rho,(y^{-1}g)^{+}}\rangle+\langle{\rho,g^{+}}\rangle}.$$
	Thus, invoking \cref{S3 lemma distance behaviors}(iii) and (iv) for $R(t)=\varepsilon(t)\sqrt{t}$ and $\theta(t)=\varepsilon(t)$ we get
	\begin{equation}\label{eq: phi quot L2 planB} \frac{\varphi_{0}(y^{-1}g)}{\varphi_{0}(g)} =e^{ \langle\rho,A(k(g)^{-1}y)\rangle}+\textrm{O}(\varepsilon(t)). 
	\end{equation}
	From \eqref{eq: heat quot L2 plan B} and \eqref{eq: phi quot L2 planB} it follows that 
	$$\frac{h_t(y^{-1}g)}{h_t(g)}=e^{ \langle\rho,A(k(g)^{-1}y)\rangle}+\textrm{O}(\varepsilon(t)), \quad g\in K\{\exp \widetilde{B}_2(t)\}K, \quad |y|<\xi,$$
	which can be in turn used in \eqref{eq: diff heat L2 planB}, so that with standard arguments as in the case $1\leq p<2$ one gets the $L^2$ convergence. The proof is now complete.
\end{proof}	

\begin{remark} In \cite{NRS24} it is proved that if $u_0\in \mathcal{L}^p(G//K)$, $1\leq p\leq 2$, then the only $z\in \mathbb{C}$ such that 
	$$\|h_t\|_p^{-1}\,\|u_0\ast h_t-z\, h_t\|_p \rightarrow 0 \quad \text{as} \quad t\rightarrow +\infty, \quad 1\leq p \leq 2.$$
	is $z=M_p=\int_{G}\textrm{d}y\, u_0(y)\, \varphi_{\pm i(2/p-1)\rho}(y)=\mathcal{H}u_0(\pm i(2/p-1)\rho)$. However, this is not the case if one allows for functions, instead of constants. Let us elaborate: let $u_0\in \mathcal{C}_c(G//K)$ and $s>0$, and consider the function
	\begin{align*}
		M_p^{s}(u_0)(g)&:=\int_{G}\textrm{d}y\,\frac{2+d(gK, yK)^s}{1+d(gK,yK)^s} \, u_0(y)\,\varphi_{\pm i(2/p-1)\rho}(y)
		\\&= M_p +\int_{G}\textrm{d}y\,\frac{1}{1+d(gK,yK)^s}\, u_0(y)\,\varphi_{\pm i(2/p-1)\rho}(y).
	\end{align*}
	Clearly, this is well-defined and bounded for $u_0\in \mathcal{L}^p(G//K)$ (observe that $|M_p^{s}(u_0)(g)|\leq 2\int_{G} |u_0|\, \varphi_{\pm i(2/p-1)\rho}$ for all $g\in G$). Thus outside the $L^p$ critical region of \cref{thm: lp critical}, a crude triangle inequality yields 
	$$\|h_t\|_p^{-1}\,\|u_0\ast h_t-M_p^{s}(u_0)(\cdot)\, h_t\|_{L^p(G\smallsetminus K\{\exp B_p(t)\}K)} \rightarrow 0 \quad \text{as} \quad t\rightarrow +\infty, \quad 1\leq p \leq 2.$$ 
	Now let us discuss convergence inside the critical region. For $g\in K\{\exp B_p(t)\}K$ and $y$ bounded, we have $d(gK, yK)\gtrsim t^{\frac{1}{2}-\varepsilon}$ for all $1\leq p \leq 2$ if $t$ is sufficiently large. In other words, $M_p^{s}(u_0)(g)=M_p+\textrm{O}\left(\mathcal{H}|u_0|(i(2/p-1)\rho) \,t^{s(-\frac{1}{2}+\varepsilon)}\right),$ when $g\in K\{\exp B_p(t)\}K$. The desired $L^p$ convergence result in $K\{\exp B_p(t)\}K$ now follows from \eqref{Ind convergence} and the triangle inequality. Altogether, we have shown that 
	$$\|h_t\|_p^{-1}\,\|u_0\ast h_t-M_p^{s}(u_0)(\cdot)\, h_t\|_p \rightarrow 0 \quad \text{as} \quad t\rightarrow +\infty, \quad 1\leq p \leq 2, \quad \forall s>0.$$

\end{remark}

\section{$L^p$ convergence, $2\leq p\leq \infty$} \label{Section.5 lp 1 to infty}

Our aim in this section is to prove the $L^p$ asymptotics \eqref{S1 Main thm 2 convergence} for $p\in [2,\infty]$. To this end, we again turn most of our attention, as in the $1\leq p<2$ case, in the $L^p$ critical region as defined in \cref{thm: lp critical}. We aim again to use heat kernel asymptotics; these are still available since $\frac{|x|}{t}\rightarrow 0$ when $x\in K\{\exp B_p(t)\}K$, according to \eqref{S2 heat kernel critical region}; however, $x^{+}$ can be near the walls for $p>2$. For the case $p=2$, we consider throughout this whole section that the $L^2$ concentration region is
\begin{align*}
	B_2(t)\,
	=\,\big\lbrace{
		H\in\overline{\mathfrak{a}^{+}}: \quad
		\varepsilon(t)\sqrt{t}\le|H|\le\tfrac{\sqrt{t}}{\varepsilon(t)}
		\,\,\,\textnormal{and}\,\,\,
		\mu(H)\ge\varepsilon(t)\sqrt{t}
	}\big\rbrace,
\end{align*}
where $\varepsilon(t)\sqrt{t}\rightarrow +\infty$, $\varepsilon(t)\rightarrow 0$.

The following lemma plays a key role in the proof of \cref{S1 Main thm 2} for $p\in [2, \infty]$.
\begin{lemma}\label{S4 Lemma ratios difference}
	Let $2\leq p\leq \infty$ and let $\delta_p(t)$ be the positive function given by 
	$\delta_2(t)=\frac{1}{\varepsilon(t)\sqrt{t}}$ if $p=2$, where $\varepsilon(t)\sqrt{t}\rightarrow +\infty$, $\varepsilon(t)\rightarrow 0$,  and given by $\delta_p(t)=\frac{R(t)}{\sqrt{t}}$ if $p\in(2,+\infty]$, where $\frac{R(t)}{\log t}\rightarrow +\infty$, $\frac{R(t)}{\sqrt{t}}\rightarrow 0$.
	Then, for  bounded $y\in{G}$ and for all $g$ in the $L^p$ critical region
	$K\{\exp B_p(t)\}K$, the following asymptotic behavior holds:
	\begin{align*}
		\frac{h_{t}(y^{-1}g)}{h_{t}(g)}
		-\frac{\varphi_{0}(y^{-1}g)}{\varphi_{0}(g)}\,
		=\,\mathrm{O}(\delta_p(t))
		\qquad\textnormal{as}\,\,\,t\longrightarrow\infty.
	\end{align*}
\end{lemma}

\begin{proof} For $p=2$, the claim has been established in \cite[Lemma 4.7]{APZ2023}. For $p\in (2,+\infty]$, a careful inspection of \cite[p.1060, Step 5]{Ank1991} reveals that one can show that
	$$h_t(g)=\left( C_0+\textrm{O}\left(\frac{R(t)}{\sqrt{t}}\right)\right)\,  t^{-\frac{\ell}{2}-|\Sigma^{++}|}\, e^{-|\rho|^2 t}\, \varphi_0(g)\, e^{-\frac{|g^{+}|^2}{4t}}$$
	(the constant $C_0$ can be computed explicitly) when $|g^{+}|\leq R(t)=\textrm{o}(\sqrt{t})$.  Thus since $y$ is bounded, we also have $|(y^{-1}g)^{+}|=\textrm{O}(R(t))$ by the triangle inequality. Therefore we can write
	\begin{align*}
		\frac{h_{t}(y^{-1}g)}{h_{t}(g)}\,
		=\,\underbrace{
			\tfrac{C_{0}+\textrm{O}\left(\frac{R(t)}{\sqrt{t}}\right) \big)}{
				C_{0}+\textrm{O}\left(\frac{R(t)}{\sqrt{t}}\right)}
		}_{1+\mathrm{O}\big(  \frac{R(t)}{\sqrt{t}} \big)}\,
		\,\frac{\varphi_{0}(y^{-1}g)}{\varphi_0(g)}\,
		e^{-\tfrac{|(y^{-1}g)^{+}|^{2}}{4t}+\tfrac{|g^{+}|^{2}}{4t}}
	\end{align*}
	and
	\begin{align*}
		\exp\{-\tfrac{|(y^{-1}g)^{+}|^{2}}{4t}+\tfrac{|g^{+}|^{2}}{4t}\}=\exp\{ \underbrace{\frac{|(y^{-1}g)^{+}|+|g^{+}|}{4t}}_{=\,\textrm{O}\left(\frac{R(t)}{t}\right)}\,(\underbrace{|g^{+}|-|(y^{-1}g)^{+}|}_{=\,\textrm{O}(1)})\}=1+\textrm{O}\left(\frac{R(t)}{t}\right).
	\end{align*}
	Hence, since by \eqref{Harnack} the quotient $\varphi_{0}(y^{-1}g)/\varphi_0(g)$ is bounded when $y$ is bounded, we conclude
	\begin{align*}
		\frac{h_{t}(y^{-1}g)}{h_{t}(g)}
		-\frac{\varphi_{0}(y^{-1}g)}{\varphi_{0}(g)}\,
		=\,\textrm{O}\left(\frac{R(t)}{\sqrt{t}}\right).
	\end{align*}
\end{proof}

We are now ready to prove the $L^p$ convergence of \cref{S1 Main thm 2}, $2\leq p \leq \infty$, for continuous and compactly supported initial data.
\begin{proof}
	Assume that $\supp u_0\subseteq K\{\exp B(0,\xi)\}K$, for some $\xi>0$. We aim to study the difference
	\begin{align}\label{eq: diff Lp2infty Cc}
		u(t,g)\,-\,M_p(u_0)(g)\,h_{t}(g)\,
		&=(u_{0}*h_{t})(g)
		\,-\,
		{h}_{t}(g)\,
		\frac{(u_{0}*\varphi_{0})(g)}{\varphi_{0}(g)}
		\notag\\[5pt]
		&=\,h_{t}(g)\,
		\int_{G}\diff{y}\,u_{0}(yK)\,
		\Big\lbrace{
			\frac{h_{t}(y^{-1}g)}{h_{t}(g)}
			-\frac{\varphi_{0}(y^{-1}g)}{\varphi_{0}(g)}
		}\Big\rbrace. 
	\end{align}
	According to the previous lemma, we have
	\begin{align}\label{eq: diff Lp2infty Cc 2}
		\frac{h_{t}(y^{-1}g)}{h_{t}(g)}
		-\frac{\varphi_{0}(y^{-1}g)}{\varphi_{0}(g)}\,
		=\,\mathrm{O}(\delta_p(t))
		\qquad\forall\,g\in K\{\exp B_p(t)\}K, \, \, 
		\forall\,y\in \textrm{supp}u_0.
	\end{align}
	Then, for $2\leq p <\infty$ the $L^p$ integral of
	${u}(t,\cdot)-M_p(u_0)(\cdot)\,{h}_{t}$	over the critical region satisfies
	\begin{align*}
		\left( \int_{K\left\{\exp B_p(t)\right\} K}
		\textrm{d}{g}\,    
		|u(t,g)\,-\,{M}_p(u_0)(g)\,{h}_{t}(g)|^p\, \right)^{1/p}
		\lesssim\,
		\delta_p(t)\,
		\underbrace{\vphantom{\Big|}
			\left(	\int_{G}\textrm{d}{g}\,{h}_{t}(g)^p \right)^{1/p}
		}_{\|h_t\|_p}\,\|u_0\|_{1},
	\end{align*}
	hence it is $\textrm{o}(\|h_t\|_p)$; meanwhile, for $p=\infty$, we get from \eqref{eq: diff Lp2infty Cc} and \eqref{eq: diff Lp2infty Cc 2} that 
	$$\|h_t\|_{\infty}^{-1}\sup_{|g|\leq R(t)}|u(t, \,g)\,-\,M_{\infty}(u_0)(g)\, h_t(g)|\lesssim\textrm{O}(\delta_{\infty}(t))\, \|u_0\|_1.$$
	
	Let us now work outside the critical region. As in the case of $1\leq p <2$, it will become evident that now the expression of the mass function $M_p(u_0)$ plays no particular role; instead, its important feature will simply be that it is bounded when $u_0\in \mathcal{C}_c(\mathbb{X})$.  
	
	We deal first with the case $p=\infty$. We have
	$$|u(t,g)\, -\, M_{\infty}(u_0)(g)\, h_t(g)|\leq \sup_{|g|> R(t)}|h_t(g)|\|u_0\|_1+\sup_{|g|> R(t)}|M_{\infty}(u_0)(g)||h_t(g)|.$$  Using the estimates \eqref{S2 global estimate phi0} for the heat kernel and \eqref{S2 global estimate phi0} for the ground spherical function, we deduce that there are $N_1, N_2>0$ such that for $|g|> R(t)$, 
	\begin{align*}
		t^{\frac{\nu}{2}}\, e^{|\rho|^2 t}\,h_t(g)&\lesssim t^{N_1}\,(1+|g|)^{N_2}\,e^{-\langle \rho, g^{+}\rangle}\lesssim t^{N_1}\,(1+|g|)^{N_2}\, e^{-\rho_{\textrm{min}}|g|} \\
		&\lesssim t^{-N}\qquad \forall N>0,
	\end{align*}
	since $\frac{R(t)}{\log t}\rightarrow +\infty$. It follows  by the boundedness of $M_{\infty}(u_0)$ that 
	$$\|h_t\|_{\infty}^{-1}\sup_{|g|>R(t)}|u(t,g)-M_{\infty}(u_0)(g)\, h_t(g)|\lesssim t^{-N} \qquad \forall N>0.$$
	
	Finally we treat the case $2\leq p <\infty$. Thus we have to check that the $L^p$ integral, $p\geq 2$
	\begin{align*}
		\left( \int_{G\smallsetminus K\left\{\exp B_p(t)\right\} K}\textrm{d}{g}\,
		|u(t,g)\,-M_p(u_0)(g)\,{h}_{t}(g)|^p\, \right)^{1/p}
		&\le\,
		\left( \int_{G\smallsetminus K\left\{\exp B_p(t)\right\} K}\textrm{d}{g}\,
		|u(t,g)|^p \right)^{1/p}\\[5pt]
		&+\,
		\left( \int_{G\smallsetminus K\left\{\exp B_p(t)\right\} K}\textrm{d}{g}\,    
		|M_p(u_0)(g)|^p \,h_{t}(g)^p  \right)^{1/p}
	\end{align*}
	is also $\textrm{o}(\|h_t\|_p)$. On the one hand, for the second summand above, we know that ${M_{p}(u_0)}$ is bounded
	and that the heat kernel ${h}_{t}$ asymptotically concentrates in the $L^p$ sense in the region
	$K\left\{\exp B_p(t)\right\} K$, hence
	\begin{align*}
		\|h_t\|_{p}^{-1}\,\int_{G\smallsetminus K\left\{\exp B_p(t)\right\} K}\textrm{d}{g}\,
		|{M}_p(u_0)(g)|^p \, {h}_{t}(g)^p\,
		\longrightarrow\,0
	\end{align*}
	as $t\rightarrow\infty$. On the other hand, we distinguish cases in $p$ for the second summand. If $p=2$, then notice that for all
	$y\in\supp{u_{0}}$ and for all $g\in{G}$ such that
	$g^{+}\notin B_2(t)$ one has
	\begin{align}\label{S4 Omega''}
		(y^{-1}g)^{+}\,\notin\,{B_2(t)}''
		=\,
		\big\lbrace{
			H\in\overline{\mathfrak{a}^{+}}\,|\,
			\varepsilon''(t)\sqrt{t}\le|H|\le\tfrac{\sqrt{t}}{\varepsilon''(t)}
			\,\,\,\textnormal{and}\,\,\,
			\mu(H)\ge\varepsilon''(t)\sqrt{t}
		}\big\rbrace
	\end{align}
	where $\varepsilon''(t)=2\varepsilon(t)$ (see the proof of \cite[Lemma 4.7]{APZ2023}). Similarly, if $p>2$, then observe that for all
	$y\in\supp{u_{0}}$ and for all $g\in B_p(t)$ then 
	\begin{align*}
		(y^{-1}g)^{+}\,\notin\,{B_p(t)}''
		=\,
		\left\{
		H\in\overline{\mathfrak{a}^{+}}: \quad |H|\leq \frac{R(t)}{2}
		\right\}, \qquad p>2.
	\end{align*}
	Hence an application of the Minkowski inequality for integrals, as in the case $1\leq p<2$, yields
	\begin{align*}
		\|u(t, \, \cdot \,)\|_{L^p(G\smallsetminus K\left\{\exp B_p(t)\right\}K)}=\textrm{o}(\|h_t\|_p).
	\end{align*}
	This concludes the proof of the heat asymptotics in $L^{p}$, $2\leq  p\leq \infty$ and for initial data $u_{0}\in\mathcal{C}_{c}(\mathbb{X})$.
\end{proof}

\subsection{$L^p$ heat asymptotics for other initial data, $2\leq p \leq \infty$}\label{Subsect other data 1 to infty}

So far we have obtained above the long-time asymptotic convergence in $L^{p}$
($2\leq p \leq \infty$) for the heat equation 
with compactly supported initial data, in the sense of \eqref{S1 Main thm 2 convergence}.
In this section we discuss convergence for further classes of initial data.

\begin{corollary}
	The $L^p$ asymptotic convergence \eqref{S1 Main thm 2 convergence}, $2\leq p\leq \infty$, still holds with initial data
	$u_{0}\in \mathcal{L}^p(G//K)$.
\end{corollary}

\begin{proof}
	The proof follows from the fact that the mass function boils down to constant, namely the spherical Fourier transform of $u_0$ evaluated at $0$ (see \cref{remark bi-K}) and a density argument, see for instance \cite{APZ2023, NRS24}.
\end{proof}

\begin{corollary}\label{Cor Lp p2infty}
	The $L^p$ asymptotic convergence \eqref{S1 Main thm 2 convergence}, $2\leq p\leq \infty$, still holds without assuming bi-$K$-invariance for $u_0$ but under the assumption $u_{0}\in L_{w_p}(\mathbb{X})$, that is, 
	\begin{align}\label{S4 other data assumption}
		\int_{G}\diff{g}\,|u_{0}(gK)|e^{\langle{\rho,g^{+}}\rangle}\,
		<\,\infty.
	\end{align}
\end{corollary}

\begin{proof}
	The proof runs as that of Corollary \ref{Cor Lp p12}, since under the assumption \eqref{S4 other data assumption} the mass function is bounded, see \cref{remark: mass bdd weight}. We omit the details.
\end{proof}

\section{Appendix}\label{Appendix}

In this section, we focus only on real hyperbolic space $\mathbb{H}^n=\mathbb{H}^n(\mathbb{R})$. We write down explicitly the mass functions $M_p$, as well as the weighted $L^1$ spaces, hoping that it will be illuminating to those who are more familiar to hyperbolic space than arbitrary rank symmetric spaces. Let us also point out that in this Appendix we will not always write down the error terms of asymptotics, for reasons of brevity. We first introduce some notation. 

We use the model of hyperbolic space given by the upper sheet of hyperboloid
\[
\mathbb{H}^n=\{ (x_0,x_1, ... ,x_n)\in\mathbb{R}^{n+1},\;x_0^2-x_1^2-...-x_n^2=1,\;x_0>0\}.\]
Using polar coordinates, we write 
\[
x_0=\cosh r\quad\text{and}\quad(x_1, ..., x_n)=(\sinh r)\,\omega\quad
\text{with}\quad r\ge 0\,\; \omega\in\mathbb{S}^{n-1}.
\]
The distance of a point $\Omega=\Omega(r,\omega)=\bigl(\cosh r,(\sinh r)\,\omega\bigr)\in\mathbb{H}^n$
to the origin $O=(1,0, ... ,0)$ of the hyperboloid  is
\[
d(\Omega,O)=r.
\]
More generally, the distance between two arbitrary points
$\Omega=\Omega(r,\omega)$ and $\Omega'=\Omega(s,\omega')$ satisfies
\begin{equation}\label{eq: dist hyperbolic}
	\cosh d(\Omega,\Omega')=(\cosh r)(\cosh s)-(\sinh r)(\sinh s)\,\omega\cdot\omega'.
\end{equation}
Denote the geodesic ball around the origin of radius $\xi>0$ by $B(O, \xi).$ The volume element is given by
\begin{equation*}
	\int_{\mathbb{H}^n}d\Omega\, f(\Omega)
	=\int_0^{\infty}dr\int_{\mathbb{S}^{n-1}}d\omega\,u_0(r,\omega)\,(\sinh r)^{n-1}.
\end{equation*}
Recall that for \textit{radial} $\mathcal{C}_c(\mathbb{H}^n)$ functions, the spherical Fourier transform is defined by 
$$\mathcal{H}u_0(\lambda)=\int_{\mathbb{H}^n}d\Omega \,u_0(\Omega)\,\varphi_{\lambda}(\Omega)=C\int_{0}^{\infty} dr\,u_0(r)\, \varphi_{\lambda}(r)\, (\sinh r)^{n-1}, \quad \lambda \in \mathbb{C},$$
where the elementary spherical function of index $\lambda$ is radial in space, even in $\lambda$, and is given by 
\begin{align}\label{eq: spherical hyp}
	\varphi_\lambda(r) & =\int_{\mathbb{S}^{N-1}} d \omega\,(\cosh r-\sinh r(\vartheta \cdot \omega))^{-i \lambda-\frac{n-1}{2}}  \notag \\
	& =\frac{\Gamma(\frac{n}{2})}{\sqrt{\pi}\Gamma(\frac{n-1}{2})} \int_0^{\pi}d \theta\, (\cosh r\pm\sinh r \cos \theta)^{\mp i \lambda-\frac{n-1}{2}}(\sin \theta)^{N-2},
\end{align}
(see for instance \cite[Proposition 3.1.4]{GaVa1988}, \cite[Ch. IV, Theorem 4.3]{Hel2000}, or \cite[p.40]{Koo84}). Notice that it holds $\varphi_{0}(r)\asymp (1+r)\,e^{-\frac{n-1}{2}r}$ globally and $\varphi_{0}(r)\sim r\,e^{-\frac{n-1}{2}r}$ as $r\rightarrow +\infty$.

The solution to the heat equation with initial data $u_0$ is given by
\begin{equation*}
	u(t,\Omega)=\int_{\mathbb{H}^n} d\Omega' \,h_t(d(\Omega,\Omega'))\,u_0(\Omega'),
\end{equation*}
where $h_t$ denotes the heat kernel on $\mathbb{H}^n$; it is a radial function, given as the inverse spherical transform of $e^{-t(\lambda^2+\frac{(n-1)^2}{4})}$.
According to \cite[p.1080]{AnJi1999}, one has the following complete asymptotics:
\begin{equation}\label{eq: heat hyperbolic}
	h_t(r)\sim\frac{1}{2\pi^{\frac{n}{2}}} \,\gamma\left(\frac{r}{2t}\right) \, t^{-\frac{3}{2}} \, r \, e^{-\frac{(n-1)^2}{4}t-\frac{n-1}{2}r-\frac{r^2}{4t}},
\end{equation}
as $t\rightarrow +\infty$ and $r\rightarrow +\infty$, where $\gamma(s)=\frac{\Gamma(s+1/2)\Gamma(s/2+(n-1)/4)}{\Gamma(s+1)\Gamma(s/2+1/4)}$.

The $L^p$ critical regions now boil down to 
$$\textbf{B}_p(t)=\left\{ \Omega \in \mathbb{H}^n: \quad \left|d(\Omega, O)-\left(\frac{2}{p}-1\right)(n-1)t\right|\leq r(t)\right\}, \quad 1\leq p<2,$$
where $r(t)$ is as in \cref{thm: lp critical}. Likewise, 
$$\textbf{B}_2(t)=\left\{\Omega \in \mathbb{H}^n: \quad r_1(t) \leq d(\Omega, O)  \leq r_2(t) \right\},$$
and 
$$\textbf{B}_p(t)=\left\{\Omega \in \mathbb{H}^n:  \quad d(\Omega, O)  \leq R(t) \right\}, \quad 2<p\leq \infty,$$
where again we use the notation of \cref{thm: lp critical} for the quantities $r_1(t), r_2(t), R(t)$.

As clear from the proof in the higher rank case, the essential ingredient is the asymptotics of the quotient $h_t(d(\Omega, \Omega'))/h_t(d(\Omega, O))$, when $\Omega \in \textbf{B}_p(t)$ and $d(\Omega',O)$ is bounded. Even further, the exponential terms will be the ones giving the crucial information, since it can be shown that $\frac{\gamma(d(\Omega, \Omega')/2t)}{\gamma(d(\Omega, O)/2t)}\sim 1$ when $d(\Omega, \Omega'), d(\Omega, O)=\textrm{O}(t)$ (see \cref{lemma bfunc ratio}).

Let us write
$\Omega=\Omega(r,\omega)\in \textbf{B}_p(t)$ and $\Omega'=\Omega(s,\omega')\in B(O, \xi)$. Then clearly
$$d(\Omega, \Omega')-d(\Omega, O)=O(1), \quad 1\leq p\leq \infty.$$
In the case $1\leq p\leq 2$, we can even be more precise, since when $t\rightarrow +\infty$ we also have $r\rightarrow +\infty$. More specifically, using the distance formula \eqref{eq: dist hyperbolic} one can compute that 
\begin{align}
	d(\Omega, \Omega')-d(\Omega, O)&=\cosh^{-1}(\cosh r\,\cosh s-\sinh r\,\sinh s\,\omega\cdot\omega')-r \notag \\
	&\sim \log(\cosh s-\sinh s \, \omega\cdot \omega') \quad \text{as} \quad t\rightarrow +\infty, \quad 1\leq p \leq 2 \label{eq: Busemann}.
\end{align}
On the other hand, we can write
\begin{align}\label{eq: diff squares}
	\frac{d(\Omega, \Omega')^2-d(\Omega, O)^2}{4t}&=(d(\Omega, \Omega')-d(\Omega, O))\,\frac{d(\Omega, \Omega')+d(\Omega, O)}{4t},
\end{align}
where the quotient $(d(\Omega, \Omega')+d(\Omega, O))/4t$ is equal to $\left(\frac{2}{p}-1\right)\frac{n-1}{2}+\textrm{O}\left(\frac{r(t)}{t}\right)$ when $1\leq p <2$, while it is $\textrm{O}\left(\frac{r_2(t)}{t}\right)$ when $p=2$, and $\textrm{O}\left(\frac{R(t)}{t}\right)$ when $2<p\leq \infty$. 

It follows from  \eqref{eq: diff squares} and \eqref{eq: Busemann} that
\begin{align*}
	e^{-\frac{n-1}{2}(d(\Omega, \Omega')-d(\Omega, O))}e^{-\frac{d(\Omega, \Omega')^2-d(\Omega, O)^2}{4t}} \sim (\cosh s-\sinh s \, \omega\cdot \omega')^{-\frac{1}{p}(n-1)}, 
\end{align*}
when $\Omega\in \textbf{B}_p(t)$, $\Omega'\in B(O,\xi)$,   $1\leq p \leq 2$. This implies by \eqref{eq: heat hyperbolic} that
$$\frac{h_t(d(\Omega, \Omega'))}{h_t(d(\Omega,O))}\sim (\cosh s-\sinh s \, \omega\cdot \omega')^{-\frac{1}{p}(n-1)}, \quad \Omega\in \textbf{B}_p(t), \, \Omega'\in B(O,\xi), \quad  1\leq p \leq 2.$$
On the other hand, it follows from \cref{S4 Lemma ratios difference} that 
$$\frac{h_t(d(\Omega, \Omega'))}{h_t(d(\Omega,O))}\sim \frac{\varphi_{0}(d(\Omega, \Omega'))}{\varphi_{0}(d(\Omega, O))}, \quad \Omega\in \textbf{B}_p(t), \, \Omega'\in B(O,\xi), \quad  2\leq p \leq \infty.$$
Thus, for $u_0\in \mathcal{C}_c(\mathbb{H}^n)$, let us define the mass function by $$M_p(u_0)(\Omega)=\int_{0}^{\infty} ds \int_{\mathbb{S}^{n-1}}d\omega' \, u_0(s,\omega')\,(\cosh s-\sinh s \, \omega\cdot \omega')^{-\frac{1}{p}(n-1)} \quad \text{if} \quad 1\leq p\leq 2,$$ and by 
$$M_p(u_0)(\Omega)=\frac{1}{\varphi_{0}(d(\Omega, O))}\int_{\mathbb{H}^n} d\Omega' \, u_0(\Omega') \,\varphi_{0}(d(\Omega, \Omega')) \quad \text{if}  \quad 2\leq p\leq \infty.$$
Then, for radial $u_0\in\mathcal{C}_c(\mathbb{H}^n)$, the mass functions boil down to
$$M_p(u_0)(\Omega)=M_p=\int_{0}^{\infty} ds\,u_0 (s)\, \varphi_{i(\frac{2}{p}-1)\frac{n-1}{2}}(s)=\mathcal{H}u_0\left(\pm\, i(\frac{2}{p}-1)\frac{n-1}{2}\right), \quad 1\leq p < 2,$$
and  
$$M_p(u_0)(\Omega)=M_p=\int_{0}^{\infty} ds\,u_0 (s)\, \varphi_{0}(s)=\mathcal{H}f\left(0\right), \quad 2\leq p\leq \infty.$$

Following the arbitrary rank approach, now we have all the ingredients to show that for all initial data $u_0\in \mathcal{L}^p_{\text{rad}}(\mathbb{H}^n)$, or $u_0$ such that $\int_{\mathbb{H}^n} d\Omega\, |u_0(\Omega)|\, w_p(\Omega) <+\infty$, where $$w_p(\Omega)=e^{\frac{n-1}{p}d(\Omega, O)} \quad  \text{if} \quad 1\leq p \leq2 \quad \text{and} \quad  w_p(\Omega)=e^{\frac{n-1}{2}d(\Omega,O)} \quad \text{if} \quad 2\leq p \leq \infty,$$ the solution $u(t, \cdot)$ satisfies
$$\|h_t\|_p^{-1} \, \|u(t, \cdot) \, -\, M_p(u_0)(\cdot) \, f\|_{L^p(\mathbb{H}^n)}\rightarrow 0, \quad \text{as} \quad t\rightarrow +\infty.$$ 
	
	\bigskip
	
	\textbf{Competing interests}
	No competing interest is declared.

	\bigskip

	\textbf{Author contributions statement}
	Not applicable.

\bigskip

	\textbf{Acknowledgments.}
	The author would like to thank J.-Ph. Anker and A. Grigor'yan for their interest in the results, discussions and encouragement. This work is funded by the Deutsche Forschungsgemeinschaft (DFG, German Research Foundation)--SFB-Gesch{\"a}ftszeichen --Projektnummer SFB-TRR 358/1 2023 --491392403.   This manuscript is partly based upon work from COST Action CaLISTA CA21109 supported by COST (European Cooperation in Science and Technology). www.cost.eu. 
	
	\bigskip
	
	

	\printbibliography
	
	\vspace{10pt}
	\address{
		\noindent\textsc{Effie Papageorgiou:}
		\href{mailto:papageoeffie@gmail.com}
		{papageoeffie@gmail.com}\\
		Institut f{\"u}r Mathematik, Universit\"at Paderborn, Warburger Str. 100, D-33098
		Paderborn, Germany}
	
\end{document}